\documentclass[11pt,a4paper]{article}
\usepackage{amsmath}
\usepackage{amsfonts}
\usepackage{amssymb}
\usepackage{centernot}
\usepackage[margin=2.8cm]{geometry}
\usepackage[colorlinks=true,linkcolor=blue,citecolor=blue,urlcolor=blue]{hyperref}
\usepackage{enumerate}
\usepackage{graphicx,fancybox}
\usepackage{algorithm2e}
\usepackage{tikz}
\usepackage{subfigure}
\usetikzlibrary{calc}

\usepackage{amsthm}
\newtheorem{theorem}{Theorem}[section]
\newtheorem{definition}{Definition}[section]
\newtheorem{proposition}{Proposition}[section]
\newtheorem{corollary}{Corollary}[section]

\newtheorem{lemma}{Lemma}[section]
\newtheorem{fact}{Fact}[section]
\newtheorem{remark}{Remark}[section]
\newtheorem{example}{Example}[section]

\newcommand{\Fix}{\operatorname{Fix}}
\newcommand{\zer}{\operatorname{zer}}

\newcommand{\epi}{\operatorname{epi}}
\newcommand{\inte}{\operatorname{int}}
\newcommand{\rint}{\operatorname{ri}}
\newcommand{\cone}{\operatorname{cone}}
\newcommand{\core}{\operatorname{core}}

\newcommand{\ran}{\operatorname{ran}}

\newcommand{\Hi}{\mathcal{H}}
\newcommand{\R}{\mathbb{R}}

\title{A new projection method for finding the closest point\\ in the intersection of convex sets}

\author{Francisco J. Arag\'on Artacho\thanks{Department of Mathematics,
University of Alicante, \textsc{Spain}. e-mail:~\url{francisco.aragon@ua.es}}
        \and Rub\'en Campoy\thanks{Department of Mathematics,
University of Alicante, \textsc{Spain}. e-mail:~\url{ruben.campoy@ua.es}}
}

\begin{document}
\maketitle

\begin{abstract}
In this paper we present a new iterative projection method for finding the closest point in the intersection of convex sets to any arbitrary point in a Hilbert space. This method, termed AAMR for averaged alternating modified reflections, can be viewed as an adequate modification of the Douglas--Rachford method that yields a solution to the best approximation problem. Under a constraint qualification at the point of interest, we show strong convergence of the method. In fact, the so-called strong CHIP fully characterizes the convergence of the AAMR method for every point in the space. We report some promising numerical experiments where we compare the performance of AAMR against other projection methods for finding the closest point in the intersection of pairs of finite dimensional subspaces.

\paragraph*{Keywords}Best approximation problem, convex set, projection, reflection, nonexpansive mapping, Douglas--Rachford algorithm, feasibility problem

\paragraph*{MSC2010:} 47H09, 47N10, 90C25
\end{abstract}

\section{Introduction}
Given two nonempty closed and convex subsets $A$, $B$ of a Hilbert space $\Hi$ and any point $q\in\Hi$, we are interested in solving the \emph{best approximation problem} of finding the closest point to $q$ in $A\cap B$; i.e.,
\begin{equation}\label{eq:general_case_problem}
\text{Find } p\in A\cap B\text{ such that } \|p-q\|=\inf_{x\in A\cap B} \|x-q\|.
\end{equation}
For any pair of parameters $\alpha,\beta\in\,]0,1[$,  we introduce the \emph{averaged alternating modified reflections operator} (\emph{AAMR operator}), which is the operator $T_{A,B,\alpha,\beta}:\Hi \mapsto \Hi$ given by
\begin{equation}\label{eq:def_GDR_0}
T_{A,B,\alpha,\beta}:=(1-\alpha)I+\alpha(2\beta P_B-I)(2\beta P_A - I),
\end{equation}
where $I$ denotes the identity mapping and $P_A$ and $P_B$ denote the \emph{projectors} (best
approximation operators) onto A and B, respectively.

Given any initial point $x_0\in\Hi$, we define a new projection method termed \emph{averaged alternating modified reflections (AAMR) method}, which is iteratively defined by
\begin{equation}\label{eq:general_case_sequence}
\doublebox{$x_{n+1}:=T_{A-q,B-q,\alpha,\beta}(x_n), \quad n=0,1,2\ldots$}
\end{equation}

If $A\cap B\neq\emptyset$, under the \emph{constraint qualification}
\begin{equation*}
q-P_{A\cap B}(q)\in (N_A+N_B)(P_{A\cap B}(q)),
\end{equation*}
where $N_A$ and $N_B$ denote the \emph{normal cones} to the sets $A$ and $B$, respectively, we shall show (Theorem~\ref{theorem:NPM_convergence}) that
the sequence generated by~\eqref{eq:general_case_sequence} weakly converges to a point $x^\star$ such that
\begin{equation*}
P_A(x^\star+q)=P_{A\cap B}(q),
\end{equation*}
and the \emph{shadow sequence} ${\left(P_A(x_n+q)\right)}_{n=0}^\infty$ is strongly convergent to $P_{A\cap B}(q)$, and thus solves problem~\eqref{eq:general_case_problem}. Even though we show that the so-called \emph{strong conical hull intersection property} (\emph{strong CHIP} in short, see Definition~\ref{def:strong_CHIP}) of $\{A,B\}$ at the point $P_{A\cap B}(q)$ is sufficient but not necessary for the convergence of the AAMR method (see Example~\ref{ex:counterex_strong_CHIP}), the strong CHIP turns out to be the precise condition to be required for the convergence of the AAMR method \emph{for every point} $q\in\Hi$ (see Theorem~\ref{theorem:NPM_convergence} and Proposition~\ref{prop:strong_CHIP}).

The AAMR operator~\eqref{eq:def_GDR_0} can be viewed as a modification of the Douglas--Rachford operator (also known as \emph{averaged alternating reflections operator}), which is the operator $DR_{A,B,\alpha}:\Hi\to\Hi$ given by
\begin{equation}\label{eq:DRM}
DR_{A,B,\alpha}:=(1-\alpha)I+\alpha(2P_B-I)(2P_A-I).
\end{equation}
The iterative method defined by the Douglas--Rachford operator is known to be weakly convergent to a point whose projection onto the set $A$ belongs to $A\cap B$ (see~\cite{LM79}), and the shadow sequence defined by the scheme is also weakly convergent to the projection onto $A$ of that point (see~\cite{Svaiter}). Surprisingly, though, the slight modification $2\beta P_B-I$ and $2\beta P_A-I$ in the \emph{reflector operators} $2P_B-I$ and $2P_A-I$ completely changes the dynamics of the sequence generated by the scheme. It permits to find, not only a point in the intersection of convex sets, but the closest point in the intersection to any arbitrary point in the space and, moreover, it forces the strong convergence of the shadow sequence.

Different projection methods have been proposed in the literature for solving the best approximation problem~\eqref{eq:general_case_problem}. For a very recent bibliography of papers and monographs on projection methods, we recommend~\cite{CC15}. Probably, the most well-known of these schemes is the method of alternating projections (MAP), which was originally introduced by John von Neumann~\cite{VN50} for solving the best approximation problem with two closed linear subspaces. For closed affine subspaces, the sequence generated by MAP is strongly convergent to the solution of~\eqref{eq:general_case_problem}. The method has been widely studied and generalized by many authors; see, e.g.,~\cite{BB93,BLY14,De83,H62,KS04,LM08,LLM09} and the monographs~\cite{BC11,C12,D01,ER11}. Although MAP is also  weakly convergent  for arbitrary convex sets (see~\cite{H04,BMR04,MR03}), it only solves the feasibility problem in this more general setting; that is, it only finds some point in the intersection of the sets, but this point does not need to be the projection onto the intersection of the point of interest. Similarly, the Douglas--Rachford method (DRM) mentioned above can be used to solve the best approximation problem for two closed affine subspaces, but for arbitrary convex sets it only finds a point in the intersection. This scheme was originally introduced in connection
with partial differential equations arising in heat conduction~\cite{DR56}. The DRM has recently gained much popularity, in part thanks to its good behavior in non-convex settings; see, e.g.,~\cite{ABglobal,ABTmatrix,ABTcomb,ABT16,BKroad,BNlocal,benoist,BS11,HLnonconvex,Plinear}. For very recent results on the behavior of the algorithm in the inconsistent case, see~\cite{BM17}.

For the general case of arbitrary convex sets, Dykstra's algorithm arose as a suitable modification of MAP that forces strong convergence to the solution of the best approximation problem, see~\cite{BB94}. It was first proposed by Dykstra in~\cite{D83} for closed and convex cones in finite-dimensional Euclidean spaces, and then extended by Boyle and Dykstra in~\cite{BD86} for closed and convex sets in a Hilbert space. For the case of affine subspaces, Dykstra's algorithm coincides with MAP (see, e.g.,~\cite[pp.~215--216]{D01}).

There are other approaches based on projection algorithms to solve best approximation problems. For instance, the Haugazeau-like algorithms introduce a projector onto the intersection of two halfspaces, which can be explicitly computed, combined in a suitable manner with another projection algorithm. This combination ensures the strong convergence of the algorithm. Haugazeau's algorithm on its basic form was first proposed in~\cite{Hau69}. Thanks to the weak-to-strong convergence principle given in~\cite{BC01}, different modifications of the method have been introduced. Another method is the one proposed by Halpern~\cite{H67}, whose strong convergence to the solution under different conditions for the parameters has been proved by different authors. The main contributions are due to Lions, Wittmann and Bauschke, see~\cite{Ba96} for details. As a result, this algorithm is sometimes called the Halpern--Lions--Wittmann--Bauschke (HLWB) method. It is also worth to mention the work of Combettes~\cite{C09}, where a Douglas--Rachford-like  strongly convergent algorithm is proposed to compute the resolvent of the sum of maximally monotone operators. Particularly, under the same constraint qualification (strong CHIP) that is needed for AAMR to converge, the scheme can be applied for solving best approximation problems. This method is discussed in Section~\ref{sec:several_sets}, where we reveal some similarities and differences with respect to AAMR. A good variety of best approximation methods has been recently collected in~\cite[Section~4.2]{BKroad}, see also~\cite[Chapter~29]{BC11}.

Observe that all the projection methods mentioned above, with the exception of Combettes' algorithm, produce an iterative sequence that converges to a point from which one can obtain the projection of the initial point of the sequence onto the intersection of the sets. Unlike in these schemes, the initial point in the AAMR method can be arbitrarily chosen in the space. Further, it is important to point out that, in general, the set of fixed points of the operator $T_{A,B,\alpha,\beta}$ is not equal to the intersection of the sets of fixed points of the operators $2\beta P_B-I$ and $2\beta P_A-I$. Therefore, the operator $T_{A,B,\alpha,\beta}$ does not belong to the broad family of operators studied in~\cite{RZ15}, see Remark~\ref{rem:Fix_T}(ii) for additional details.

The paper is organized as follows. Some preliminary concepts and auxiliary results are presented in Section~\ref{sec:Preliminaries}. In Section~\ref{sec:operator} we analyze the main properties of the AAMR operators. We introduce the new projection scheme in Section~\ref{sec:method}, where our main convergence results are collected. In Section~\ref{sec:several_sets} we show how finitely many sets can be handled through a standard product space formulation. Various numerical experiments performed on finite-dimensional subspaces are presented in Section~\ref{sec:numerical}. Finally, conclusions and future work are drawn in  Section~\ref{sec:conclusion}.

\section{Preliminaries}\label{sec:Preliminaries}

Throughout this paper our setting is the real Hilbert space $\Hi$  equipped with the inner product~$\langle\cdot,\cdot\rangle$ and the induced norm~$\|\cdot\|$. We abbreviate \emph{norm convergence} of sequences in $\Hi$ with $\to$  and we use $\rightharpoonup$ for \emph{weak convergence}. The \emph{range} of an operator $T$ is denoted by $\ran T:=T(\Hi)$ (with closure $\overline{\ran }\ T$), the set of \emph{fixed points} of $T$ is denoted by $\Fix T:=\{x\in\Hi \,|\, x\in T(x)\}$, and its set of zeros by $\zer T:=\{x\in\Hi \,|\, 0\in T(x)\}$. For a subset $A$ of $\Hi$, we denote by $\inte A$, $\rint A$, $\core A$ and $\cone A$ the \emph{interior} of $A$, the \emph{relative interior} of $A$, the \emph{algebraic interior} of $A$ and the \emph{cone generated} by $A$, respectively; i.e.,
\begin{gather*}
\core A=\left\{a\in A\mid \forall x\in\Hi, \exists\varepsilon>0\text{ such that } a+\lambda x\in A,\, \forall \lambda\in[-\varepsilon,\varepsilon]\right\},\\
\cone A=\left\{\lambda a\mid\lambda\geq 0, a\in A\right\}.
\end{gather*}
We denote by $\overline{A}$ and $A^\bot$ the closure and the orthogonal complement of the set $A$, respectively; i.e.,
$$
A^\bot=\{x\in\Hi \mid \langle a,x\rangle =0, \forall a\in A \}.
$$

Given a nonempty subset $C\subseteq\Hi$ and $x\in \Hi$, a point $p\in C$ is said to be a \emph{best approximation} to $x$ from $C$ if

\begin{equation*}
\|p-x\|=d(x,C):=\inf_{c\in C}\|c-x\|.
\end{equation*}

If a best approximation in $C$ exists for every point in $\Hi$, then $C$ is \emph{proximal}. If every point $x\in\Hi$ has exactly one best approximation $p$, then $C$ is \emph{Chebyshev} and $p$ is called the \emph{projection} of $x$ onto $C$. In this case, the \emph{projector} is the operator $P_C$ that maps every $x\in\Hi$ to its unique projection onto $C$, that is $P_C(x)=p$.

\begin{fact}\label{fact:projection}
	Let $C\subseteq\Hi$ be nonempty, closed and convex. Then the following hold.
	\begin{itemize}
		\item[(i)] $C$ is Chebyshev.
		\item[(ii)] For every $x\in\Hi$,
		$$
		p=P_C(x) \Leftrightarrow p\in C  \text{ and }  \langle c-p,x-p \rangle \leq 0 \text{ for all } c\in C.
		$$
		\item[(iii)] For every $x\in\Hi$ and $\lambda\geq 0$,
		$$
		P_C\left(P_C(x)+\lambda\left(x-P_C(x)\right)\right)=P_C(x).
		$$
		\item[(iv)] For every $y\in\Hi$, $P_{y+C} (x) = y + P_C (x - y)$.
		\item[(v)] For every $\lambda\in\R$, $P_{\lambda C}(\lambda x)=\lambda P_C(x)$.
	\end{itemize}
\end{fact}
\begin{proof}
	See, e.g., \cite[Theorem~3.14, Proposition~3.17 and Proposition~3.19]{BC11} and \cite[2.7]{De83}.
\end{proof}

\begin{definition}
	Let $D$ be a nonempty subset of $\Hi$ and let $T:D\mapsto\Hi$. The operator $T$ is said to be
	\begin{itemize}
		\item[(i)] \emph{nonexpansive} if
		\begin{equation*}
		\|T(x)-T(y)\|\leq\|x-y\|, \quad \forall x,y\in D;
		\end{equation*}
		\item[(ii)] \emph{firmly nonexpansive} if
		\begin{equation*}
		\|T(x)-T(y)\|^2+\|(I-T)(x)-(I-T)(y)\|^2\leq\|x-y\|^2, \quad \forall x,y\in D,
		\end{equation*}
		or, equivalently,
		\begin{equation*}
		\langle x-y,T(x)-T(y)\rangle\geq \|T(x)-T(y)\|^2, \quad \forall x,y\in D;
		\end{equation*}
		\item[(iii)] $\gamma$-\emph{cocoercive} for $\gamma>0$ if $\gamma T$ is firmly nonexpansive, i.e.,
		\begin{equation*}
		\langle x-y,T(x)-T(y)\rangle\geq \gamma\|T(x)-T(y)\|^2, \quad \forall x,y\in D;
		\end{equation*}
		\item[(iv)] \emph{contractive} if there exists some constant $0\leq \kappa<1$ such that
		\begin{equation*}
		\|T(x)-T(y)\|\leq\kappa\|x-y\|, \quad \forall x,y\in D;
		\end{equation*}
		\item[(v)] \emph{quasi-nonexpansive} if
		\begin{equation*}
		\|T(x)-y\|\leq\|x-y\|, \quad \forall x\in D,\; \forall y\in\Fix T;
		\end{equation*}
		\item[(vi)] \emph{strictly quasi-nonexpansive} if
		\begin{equation*}
		\|T(x)-y\|<\|x-y\|, \quad \forall x\in D \backslash\Fix T,\; \forall y\in\Fix T;
		\end{equation*}
		\item[(vi)] $\alpha$\emph{-averaged} for  $\alpha\in\,]0,1[$, if there exists a nonexpansive operator $R:D\mapsto\Hi$ such that
		\begin{equation*}
		T=(1-\alpha)I+\alpha R.
		\end{equation*}
	\end{itemize}
\end{definition}

\begin{remark}
	Firm nonexpansiveness implies nonexpansiveness, which itself implies quasi-nonexpan\-siveness. The converse implications are not true. For more, see~\cite[Chapter~4]{BC11}.
\end{remark}

\begin{fact}\label{fact:firmly_nonexp}
	Let $D$ be a nonempty subset of $\Hi$ and let $T:D\mapsto\Hi$. The following hold:
	\begin{itemize}
		\item[(i)] $T$ is firmly nonexpansive $\Leftrightarrow$ $2T-I$ is nonexpansive.
		\item[(ii)] If $T$ is $\alpha$-averaged, then $T$ is nonexpansive and strictly quasi-nonexpansive. Moreover, if $\alpha\in\left]0,\frac{1}{2}\right]$ then $T$ is firmly nonexpansive.
	\end{itemize}
\end{fact}
\begin{proof}
	See, e.g., \cite[Proposition~4.2, Remark~4.24, Remark~4.26 and Remark~4.27]{BC11}.
\end{proof}

\begin{fact}\label{fact:projection_firmly_nonexapnsive}
	Let $C\subseteq\Hi$ be nonempty, closed and convex. Then the projector operator $P_C$ is firmly nonexpansive. Moreover, if $C$ is a closed subspace, then $P_C$ is a linear mapping.
\end{fact}
\begin{proof}
	See, e.g., \cite[Proposition~4.8]{BC11} and~\cite[Theorem~5.13]{D01}.
\end{proof}

\begin{fact}\label{fact:convex_closed_FixT}
	Let $D\subseteq\Hi$ be nonempty, closed and convex, and let $T:D\mapsto\Hi$ be nonexpansive. Then, $\Fix T$ is a closed and convex set.
\end{fact}
\begin{proof}
	See, e.g., \cite[Corollary~4.15]{BC11}.
\end{proof}

In the next result we consider a Krasnosel'ski\u{\i}--Mann iteration. The second part is a straightforward consequence of~\cite[Theorem~2.1]{C09}, which is a refinement of the algorithm proposed by Lions and Mercier in~\cite{LM79}.

\begin{fact}\label{fact:convergence}
	Let $T_1,T_2:\Hi\mapsto \Hi$ be firmly nonexpansive operators, let ${(\lambda_n)}_{n=0}^\infty$ be a sequence in $[0,1]$, and let $x_0\in\Hi$. Consider   $T:=(2T_2-I)(2T_1-I)$ and suppose $\Fix T\neq\emptyset$. Set
	$$
	x_{n+1}=(1-\lambda_n)x_n+\lambda_nT(x_n)\quad\text{for }n=0,1,2\ldots.
	$$
	Then the following hold:
	\begin{enumerate}
		\item[(a)] If $\sum_{n\geq 0} \lambda_n(1-\lambda_n)=+\infty$ for all $n\geq 0$, then
		\begin{enumerate}
			\item[(i)] $({x_{n+1}-x_n)}_{n=0}^\infty$ converges strongly to 0.
			\item[(ii)] ${(x_n)}_{n=0}^\infty$ converges weakly to a point in $\Fix T$.
		\end{enumerate}
		\item[(b)] Suppose that $T_1$ is $\gamma$-cocoercive for some $\gamma>1$ and $\inf_{n\geq 0}{\lambda_n}>0$. Then $\left({T_1(x_n)}\right)_{n=0}^\infty$ converges strongly to the unique point in $T_1(\Fix T)$.
	\end{enumerate}
\end{fact}
\begin{proof}
(a) This is a Krasnosel'ski\u{\i}--Mann algorithm (see, e.g.,~\cite[Theorem~5.14]{BC11}), and $T$ is nonexpansive by Fact~\ref{fact:firmly_nonexp}.\\
(b)	Since $T_1$ and $T_2$ are firmly nonexpansive, by~\cite[Corollary~23.8]{BC11}, the operators $A_i:=T_i^{-1}-I$ are maximally monotone (see e.g.~\cite[Definition~20.20]{BC11}) and satisfy $T_i=J_{A_i}$, for $i=1,2$, where $J_{A_i}=(I+A_i)^{-1}$ is the \emph{resolvent} of $A_i$. By~\cite[Proposition~25.1]{BC11}, we have $\zer(A_1+A_2)=J_{A_1}(\Fix((2J_{A_2}-I)(2J_{A_1}-I)))=T_1(\Fix T)\neq \emptyset$. By assumption, $T_1$ is $\gamma$-cocoercive for some $\gamma>1$. Then, by~\cite[Proposition~23.11]{BC11}, we know that $A_1$ is $(\gamma-1)$-\emph{strongly monotone} (i.e., $A_1-(\gamma-1)I$ is monotone). Hence, (b) is a direct consequence of~\cite[Theorem~2.1(ii)(b)]{C09}.
\end{proof}

\begin{fact}\label{fact:linearmap_strong_convergence}
	Let $T:\Hi\to\Hi$ be a nonexpansive linear operator and let $x_0\in\Hi$. Set $x_{n+1}=T(x_n)$, $n=0,1,2\dots$ Then
	$$
	x_n\rightarrow P_{\Fix T}(x_0) \Leftrightarrow x_n-x_{n+1}\rightarrow0.
	$$
\end{fact}
\begin{proof}
	See, e.g., \cite[Proposition~5.27]{BC11}.
\end{proof}

\begin{fact}\label{fact:pazy}
	Given $\alpha\in\,]0,1[$, let $T:\Hi\mapsto\Hi$ be an $\alpha$-averaged operator. For any $x\in\Hi$, the following hold:
	\begin{itemize}
		\item[(i)] $(T^n(x)-T^{n+1}(x))_{n=0}^\infty$ converges in norm to the unique element of minimum norm in $\overline{\ran}(I-T)$;
		\item[(ii)] $\Fix T=\emptyset \Leftrightarrow \|T^n(x)\|\rightarrow +\infty$.
	\end{itemize}
\end{fact}
\begin{proof}
	(i) See \cite[Corollary~2.3]{BBR78}, \cite[Corollary~2]{P71}. (ii) See \cite[Corollary~2.2]{BBR78}.
\end{proof}

Let  $C\subset\Hi$  be a nonempty convex set and let $x\in\Hi$. The \emph{normal cone} mapping to $C$ is given by
$$
N_C(x):x\mapsto\left\{\begin{array}{ll}\{u\in\Hi \mid \langle u, c-x \rangle\leq 0, \quad \forall c\in C \} &\text{if }x\in C,\\
\emptyset & \text{otherwise.}\end{array}\right.
$$
The nearest point projection can be characterized by the normal cone.

\begin{fact}\label{fact:projection_normalcone}
	Let $C\subseteq\Hi$ be a nonempty closed and convex set, and let $x$ and $p$ be points in~$\Hi$. Then,
	$$
	p=P_C(x) \Leftrightarrow x-p\in N_C(p).
	$$
\end{fact}
\begin{proof}
	See, e.g., \cite[Proposition~6.46]{BC11}.
\end{proof}

The following notion, coined by Chui, Deutsch and Ward in~\cite{CDW90,CDW92} and developed by Deutsch, Li and Ward in~\cite{DLW97}, has been widely studied in the literature, see also~\cite{D01}.

\begin{definition}\label{def:strong_CHIP}
	Let $C$ and $D$ be two closed and convex subsets of $\Hi$. The pair of sets $\{C,D\}$ is said to have the \emph{strong conical hull intersection property} (or the \emph{strong CHIP}) at $x\in C\cap D$ if
	$$
	N_{C\cap D}(x)=N_C(x)+N_D(x).
	$$
	We say $\{C,D\}$ has the strong CHIP if it has the strong CHIP at each $x\in C\cap D$.
\end{definition}

For relationships between strong CHIP and the so-called \emph{bounded linear regularity} property in Euclidean spaces, which plays an important role in the rate of convergence of projection algorithms, see~\cite{BBT00,KLT16}.

Next we show a sufficient condition for the strong CHIP in terms of the epigraph of the support function.
Recall that, given a nonempty subset $C$ of~$\Hi$, the \emph{support function} $\sigma_C$ is defined by $\sigma_C(x) = \sup_{c\in C} \langle c,x\rangle$ for $x\in\Hi$. The \emph{epigraph} of a function $f:\Hi\to\mathbb{R}\cup\{+\infty\}$ is the set $\epi f$ defined by
$$\epi f=\{(x,r)\in\Hi\times\mathbb{R}\mid f(x)\leq r\}.$$

\begin{fact}\label{fact:normal_cone_intersection_formula}
	Let $C$ and $D$ be two closed and convex subsets of $\Hi$. Then $\{C,D\}$ has the strong CHIP if one of the following conditions hold:
	\begin{itemize}
		\item[(i)] If the set $\epi\sigma_C+ \epi\sigma_D$ is weakly closed (which holds e.g. if $(\inte D)\cap C\neq\emptyset$, $0\in \core(C-D)$ or $\cone(C-D)$ is a closed subspace);
        \item[(ii)] If $\Hi$ is finite dimensional and $(\rint C)\cap(\rint D)\neq\emptyset$.
    \end{itemize}
\end{fact}
\begin{proof}
	See~\cite[Theorem~3.1, Proposition~3.1]{BJ05}, \cite[Corollary~23.8.1]{RCK70}.
\end{proof}

Finally, we present some useful results that characterize the strong CHIP for closed subspaces.

\begin{fact}\label{fact:subspaces_normalcone}
	Let $M_1, M_2\subseteq\Hi$ be closed subspaces. Then the following hold:
	\begin{itemize}
		\item[(i)] For all $x\in M_1$, one has $N_{M_1}(x)=M_1^\bot$;
		\item[(ii)] $\left(M_1\cap M_2\right)^\bot=\overline{M_1^\bot + M_2^\bot}$.
	\end{itemize}
\end{fact}
\begin{proof}
	See, e.g., \cite[Theorem~4.5 and Theorem~4.6]{D01}.
\end{proof}

\begin{definition}\label{def:F_angle}
	Let $U,V$ be two closed subspaces in $\Hi$. The \emph{Friedrichs angle} between $U$ and $V$ is the angle in $[0,\frac{\pi}{2}]$ whose cosine is
	$$
	c_F(U,V):=\sup\left\{|\langle u, v \rangle| \, : \, u\in U\cap(U\cap V)^\bot, v\in V\cap(U\cap V)^\bot, \|u\|\leq 1,\|v\|\leq 1 \right\}.
	$$
\end{definition}

\begin{fact}\label{fact:chip_subspaces}
	Let $U,V$ be two closed subspaces in $\Hi$. Then
    $$\{U,V\}\text{ has the strong CHIP}\iff U+V\text{ is closed}\iff U^\bot+V^\bot\text{ is closed}\iff c_F(U,V)<1.$$
	In particular, the latter holds if $U$ or $V$ has finite dimension or finite codimension.
\end{fact}
\begin{proof}
	See, e.g., \cite[Therorem~9.35 and Corollary~9.37]{D01} and Fact~\ref{fact:subspaces_normalcone}.
\end{proof}

\section{The averaged alternating modified reflections operator}\label{sec:operator}

We begin this section with the following simple result that motivates the definition of what we call a \emph{modified reflection}.
\begin{proposition}\label{prop:1}
	Let $D$ be a nonempty subset of $\Hi$ and let $T:D\mapsto\Hi$. If $T$ is firmly nonexpansive, then $2\beta T-I$ is nonexpansive for any $\beta\in\,]0,1]$.
\end{proposition}
\begin{proof}
	Since $T$ is firmly nonexpansive, the operator $\beta T$ is firmly nonexpansive for any $\beta\in\,]0,1]$. The result follows from Fact~\ref{fact:firmly_nonexp}(i).
\end{proof}

\begin{definition}
	Let $C$ be a nonempty closed convex set. Given any $\beta\in\,]0,1]$, the operator $2\beta P_C-I$ is called a \emph{modified reflector} operator. The case $\beta=1$ is known as the \emph{reflector} and is denoted by $R_C:=2P_C-I$.
\end{definition}

\begin{remark}\label{rem:mod_ref_nonexp}
	For any $\beta\in\,]0,1]$ and any $x\in\Hi$, one has
	$$(2\beta P_A-I)(x)=\beta R_A(x)+\beta x-x=\beta R_A(x)+(1-\beta)(-x);$$
	that is, $(2\beta P_A-I)(x)$ is a convex combination of $R_A(x)$ and $-x$ (see Figure~\ref{fig:T_alphabeta}).
\end{remark}

The next result shows that the modified reflector operators have a unique fixed point.

\begin{proposition}\label{prop:fixed_mod_ref}
	Let $C\subseteq \Hi$ be nonempty, closed and convex, and let $\beta\in\,]0,1[$. Then
	$$
	\Fix(2\beta P_C-I)=\left\{ \beta P_C(0) \right\}.
	$$
\end{proposition}
\begin{proof}
	First, notice that $y\in\Fix(2\beta P_C-I)$ if and only if $\beta P_C(y)=y$; that is, the set of $\Fix(2\beta P_C-I)=\Fix(\beta P_C)$. Since $\beta P_C(0)=P_C(0)+(1-\beta)(0-P_C(0))$, with $1-\beta>0$, we deduce from Fact~\ref{fact:projection}(iii) that $P_C(\beta P_C(0))=P_C(0)$. Thus, $\beta P_C(0)\in\Fix(2\beta P_C-I)$. The uniqueness directly follows from the fact that $\beta P_C$ is a contraction, by Fact~\ref{fact:projection_firmly_nonexapnsive}.
\end{proof}

\begin{definition}
	Let $A,B\subseteq\Hi$ be nonempty, closed and convex sets. Given $\alpha,\beta\in\,]0,1[$,  we define the \emph{averaged alternating modified reflections (AAMR) operator} $T_{A,B,\alpha,\beta}:\Hi \mapsto \Hi$ as
	\begin{equation}\label{eq:def_GDR}
	T_{A,B,\alpha,\beta}:=(1-\alpha)I+\alpha(2\beta P_B-I)(2\beta P_A - I).
	\end{equation}
	Where there is no ambiguity, we will abbreviate the notation of the operator $T_{A,B,\alpha,\beta}$ by~$T_{\alpha,\beta}$.	
\end{definition}

\begin{proposition}\label{prop:nonexpansiveness_gdr}
	If $A,B\subseteq\Hi$ are nonempty, closed and convex sets, then $T_{\alpha,\beta}$ is $\alpha$-averaged for all $\alpha,\beta\in\,]0,1[$, and thus nonexpansive and strictly quasi-nonexpansive. Moreover, if $\alpha\in\,]0,1/2]$, then $T_{\alpha,\beta}$ is firmly nonexpansive.
\end{proposition}
\begin{proof}
	It is straightforward, in view of Proposition~\ref{prop:1} and Fact~\ref{fact:firmly_nonexp}(ii).
\end{proof}

A geometric interpretation of the AAMR operator $T_{A,B,\alpha,\beta}$ is shown in Figure~\ref{fig:T_alphabeta}, which was created with \texttt{Cinderella}~\cite{Cinderella}. It is important to emphasize that we require $\beta<1$ in the definition of the AAMR operator. The case $\beta=1$ in~\eqref{eq:def_GDR} corresponds with the Douglas--Rachford operator $DR_{A,B,\alpha}=(1-\alpha)I+\alpha R_B R_A$, whose behavior is remarkably different. In particular, one has $P_A(\Fix DR_{A,B,\alpha})=A\cap B$, while $P_A(\Fix T_{\alpha,\beta})\subsetneq A\cap B$, as we show in the next remark.

\begin{figure}[ht]
	\begin{center}
		\includegraphics[width=0.95\textwidth]{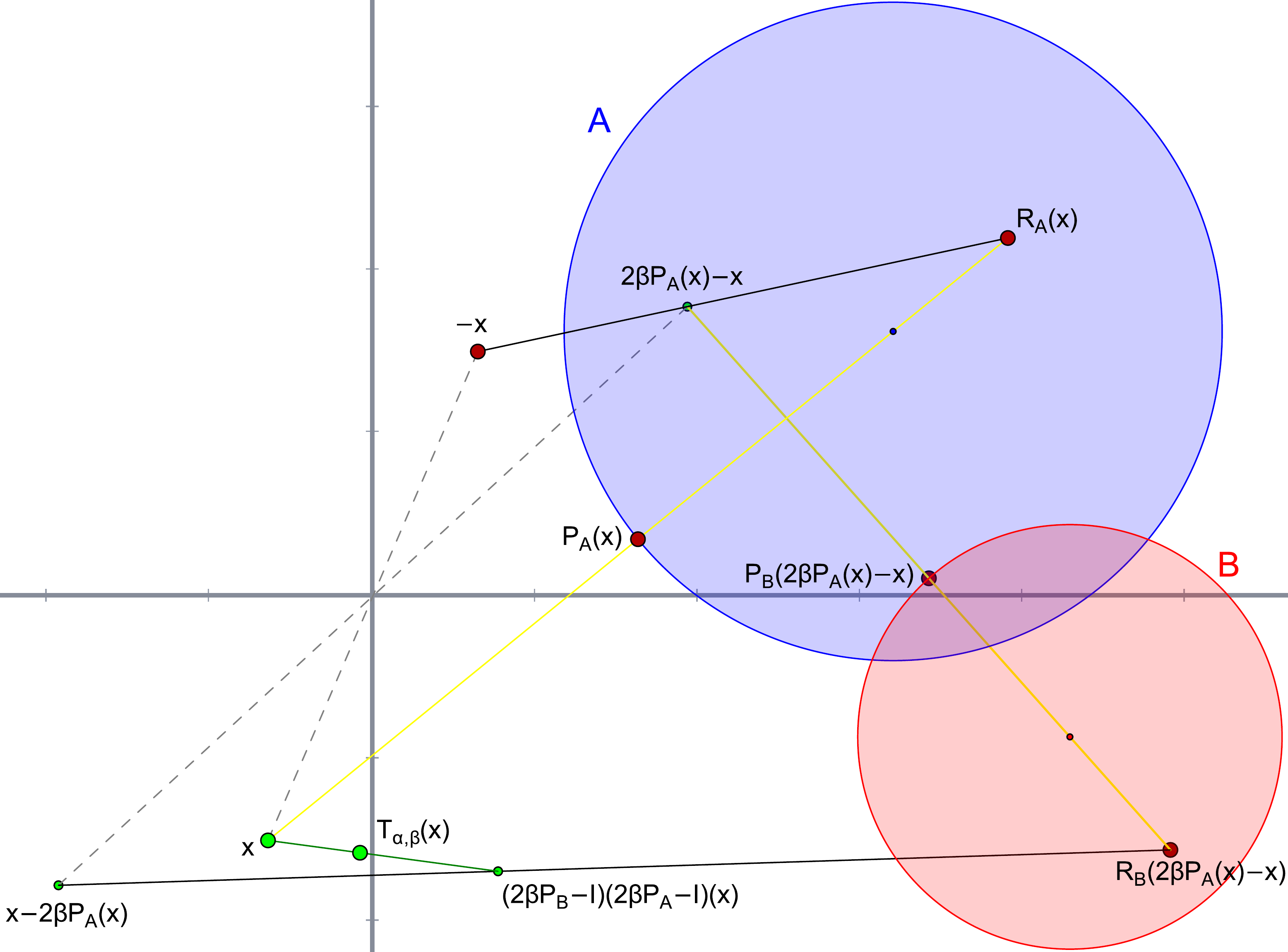}
	\end{center}
	\caption{Geometric interpretation of the modified reflector operator $2\beta P_A-I$ and the AAMR operator $T_{\alpha,\beta}$.}\label{fig:T_alphabeta}
\end{figure}

\begin{remark}\label{rem:Fix_T}
	\emph{(i)} Observe that $\Fix T_{\alpha,\beta}=\Fix((2\beta P_B-I)(2\beta P_A-I))$ for all $\alpha\in\,]0,1[\,$. In fact, $x\in\Fix T_{\alpha,\beta}$ if and only if
	\begin{align*}
	x&=T_{\alpha,\beta}( x)=(1-\alpha)x+\alpha(2\beta P_B-I)(2\beta P_A( x)-x)\\
	&=(1-\alpha)x+2\alpha\beta P_B(2\beta P_A(x)-x)-\alpha(2\beta P_A( x)-x)\\
	&=x+2\alpha\beta P_B(2\beta P_A( x)-x)-2\alpha\beta P_A(x);
	\end{align*}
	that is,
	\begin{equation}\label{eq:fixed_point}
	x\in\Fix T_{\alpha,\beta} \Leftrightarrow P_B(2\beta P_A(x)-x)=P_A(x).
	\end{equation}
	As a consequence, we have $P_A(\Fix T_{\alpha,\beta})\subset A\cap B$. In general, though, $P_A(\Fix T_{\alpha,\beta})\neq A\cap B$. For a simple example, consider $A:=\mathbb{R}^2$ and $B:=\{(x_1,x_2)\in\mathbb{R}^2\mid x_1=1\}$, and choose any $\beta\in\,]0,1[\,$. Then, it can be easily checked using~\eqref{eq:fixed_point} that $\Fix T_{\alpha,\beta}=\{(1,0)\}$ and
	$$P_A(\Fix T_{\alpha,\beta})=P_A(\{(1,0)\})=\{(1,0)\}\neq A\cap B=B.$$
	
	\emph{(ii)} 	As a consequence of Proposition~\ref{prop:fixed_mod_ref}, $\Fix(2\beta P_A-I)\cap\Fix(2\beta P_B-I)=\emptyset$ whenever $P_A(0)\neq P_B(0)$. Hence, in general, $$\Fix T_{\alpha,\beta}=\Fix\left((2\beta P_B-I)(2\beta P_A-I)\right)\neq \Fix(2\beta P_A-I)\cap\Fix(2\beta P_B-I).$$ For instance, consider the same example as in~\emph{(i)}. By Proposition~\ref{prop:fixed_mod_ref}, we have $\Fix(2\beta P_A-I)=\{(0,0)\}$ and $\Fix(2\beta P_B-I)=\{(\beta,0)\}$, while $\Fix T_{\alpha,\beta}=\{(1,0)\}$.
	
	Therefore, the operator $T_{A,B,\alpha,\beta}$ does not belong to the broad family of operators studied by Reich and Zalas in~\cite{RZ15}, which covers many projection algorithms, because they consider the general problem of finding $x\in \bigcap_{i=1}^n\Fix U_i\neq\emptyset$ for some quasi-nonexpansive operators $U_i:\Hi\to\Hi$.
\end{remark}

The following result shows that, in fact, the fixed points of the AAMR operator $T_{\alpha,\beta}$ are very special.

\begin{proposition}\label{prop:projected_fixT}
	Let $A,B\subseteq\Hi$ be nonempty, closed and convex sets, and let $\alpha,\beta\in\,]0,1[$. If $x\in \Fix T_{\alpha,\beta}$, then $A\cap B\neq \emptyset$ and
	$$
	P_A(x)=P_{A\cap B}(0).
	$$
\end{proposition}

\begin{proof}
	If $x\in \Fix T_{\alpha,\beta}$, we know by~\eqref{eq:fixed_point} that
	\begin{equation*}
	P_A(x)=P_B(2\beta P_A(x)-x),
	\end{equation*}
	which implies $P_A(x)\in A\cap B$.
	Using twice the characterization of the projections given in Fact~\ref{fact:projection}(ii), we obtain
	\begin{gather}
	\langle y-P_A(x),x-P_A(x) \rangle \leq  0, \quad \forall y\in A, \label{eq:prop:projected_fixT_aux1}\\
	\langle y-P_A(x),2\beta P_A(x)-x-P_A(x) \rangle \leq 0, \quad \forall y\in B. \label{eq:prop:projected_fixT_aux2}
	\end{gather}
	Inequalities~\eqref{eq:prop:projected_fixT_aux1} and~\eqref{eq:prop:projected_fixT_aux2} hold simultaneously for any $y\in A\cap B$. Then, by adding them, we deduce
	\begin{equation*}
	\langle y-P_A(x),-2(1-\beta)P_A(x) \rangle \leq 0, \quad \forall y\in A\cap B.
	\end{equation*}
	As $\beta < 1$, the factor $2(1-\beta)$ is strictly positive and can be removed. Therefore,
	\begin{equation*}
	\langle y-P_A(x),-P_A(x) \rangle \leq 0, \quad \forall y\in A\cap B.
	\end{equation*}
	By Fact~\ref{fact:projection}(ii), we conclude that $P_A(x)=P_{A\cap B}(0)$.
\end{proof}	

In the next theorem we present a \emph{constraint qualification} that characterizes the nonemptiness of the set of fixed points of the AAMR operators.

\begin{theorem}\label{th:fixT_cone}
	Let $A,B\subseteq \Hi$ be nonempty closed and convex sets, and let $\alpha,\beta\in\,]0,1[$. Then,
	$$
	\Fix T_{\alpha,\beta}\neq\emptyset \Leftrightarrow A\cap B\neq\emptyset\text{ and }-P_{A\cap B}(0)\in (N_A+N_B)\left(P_{A\cap B}(0)\right).
	$$
\end{theorem}
\begin{proof}
	To prove the direct implication, pick any $x\in \Fix T_{\alpha,\beta}$. Then, by Proposition~\ref{prop:projected_fixT}, we have $A\cap B\neq\emptyset$ and
	\begin{equation*}
	P_{A\cap B}(0)=P_A(x)=P_B(2\beta P_A(x)-x).
	\end{equation*}
	Thus, by~Fact~\ref{fact:projection_normalcone}, we deduce
	$$x-P_{A\cap B}(0)\in N_A(P_{A\cap B}(0))$$
	and
	$$(2\beta-1)P_{A\cap B}(0)-x=2\beta P_A(x)-x-P_{A\cap B}(0)\in N_B(P_{A\cap B}(0)).$$
	By taking $d_A:=\frac{1}{2(1-\beta)}\left(x-P_{A\cap B}(0)\right)$ and $d_B:=\frac{1}{2(1-\beta)}\left((2\beta-1) P_{A\cap B}(0) -x\right)$, we get
	\begin{equation*}
	-P_{A\cap B}(0)=d_A+d_B,
	\end{equation*}
	with $d_A\in N_A(P_{A\cap B}(0))$ and $d_B\in N_B(P_{A\cap B}(0))$, as claimed.
	
	To prove the converse implication, assume that $A\cap B\neq\emptyset$, and let $d_A\in N_A(P_{A\cap B}(0))$ and $d_B\in N_B(P_{A\cap B}(0))$ be such that
	\begin{equation}\label{eq:prop_fixT_cone_aux1}
	-P_{A\cap B}(0)=d_A+d_B.
	\end{equation}
	Take
	\begin{equation}\label{eq:x_dA_dB}
	x:=P_{A\cap B}(0)+2(1-\beta)d_A.
	\end{equation}
	As $\beta<1$, we have $2(1-\beta)d_A\in N_A(P_{A\cap B}(0))$. Then, by~Fact~\ref{fact:projection_normalcone}, we get
	\begin{equation}\label{eq:prop_fixT_cone_aux2}
	P_A(x)=P_A\left(P_{A\cap B}(0)+2(1-\beta)d_A\right)=P_{A\cap B}(0).
	\end{equation}
	Hence,
	\begin{equation}\label{eq:prop_fixT_cone_aux3}
	\begin{aligned}
	2\beta P_A(x)-x & =  2\beta P_{A\cap B}(0) - P_{A\cap B}(0)-2(1-\beta)d_A\\
	& =  P_{A\cap B}(0) + 2(1-\beta)(-P_{A\cap B}(0)-d_A).
	\end{aligned}
	\end{equation}
	Now, by combining~\eqref{eq:prop_fixT_cone_aux1} and~\eqref{eq:prop_fixT_cone_aux3}, we have
	\begin{equation}\label{eq:prop_fixT_cone_aux4}
	2\beta P_A(x)-x =  P_{A\cap B}(0) +2(1-\beta)d_B.
	\end{equation}
	Then, we use again~Fact~\ref{fact:projection_normalcone} in~\eqref{eq:prop_fixT_cone_aux4} to obtain
	\begin{equation}\label{eq:prop_fixT_cone_aux5}
	P_B(2\beta P_A(x)-x)=P_{A\cap B}(0).
	\end{equation}
	Finally, from \eqref{eq:prop_fixT_cone_aux2} and \eqref{eq:prop_fixT_cone_aux5}, we deduce
	\begin{equation}\label{eq:prop_fixT_cone_aux6}
	P_A(x)=	P_B(2\beta P_A(x)-x),
	\end{equation}
	which implies $x\in\Fix T_{\alpha,\beta}$, by~\eqref{eq:fixed_point}.
\end{proof}

The following corollary is a direct consequence of Theorem~\ref{th:fixT_cone} and characterizes the nonemptiness of the set of fixed points of the AAMR operators defining our iterative methods.

\begin{corollary}\label{cor:NPM_fixT_cone}
	Let $A,B\subseteq \Hi$ be nonempty closed and convex sets, and let $\alpha,\beta\in\,]0,1[$. Then for any $q\in\Hi$,
	$$
	\Fix T_{A-q,B-q,\alpha,\beta}\neq\emptyset \Leftrightarrow A\cap B\neq\emptyset\text{ and }q-P_{A\cap B}(q)\in (N_A+N_B)\left(P_{A\cap B}(q)\right).
	$$
\end{corollary}
\begin{proof}
	According to Theorem~\ref{th:fixT_cone}, $\Fix T_{A-q,B-q,\alpha,\beta}\neq\emptyset$ if and only if,
	$$
	(A-q)\cap (B-q)\neq\emptyset\quad\text{and}\quad -P_{(A-q)\cap (B-q)}(0)\in (N_{A-q}+N_{B-q})\left(P_{(A-q)\cap (B-q)}(0)\right).
	$$
	Observe that $(A-q)\cap (B-q)=(A\cap B) -q$. Then $(A-q)\cap (B-q)\neq\emptyset \Leftrightarrow A\cap B\neq\emptyset$.
	Now, by Fact~\ref{fact:projection}(iv), we have
	\begin{equation*}
	-P_{(A\cap B)-q}(0)=-P_{A\cap B}(q)+q,
	\end{equation*}
	and then,
	\begin{align*}
	(N_{A-q}+N_{B-q})\left(P_{(A\cap B)-q}(0)\right)&=(N_{A}+N_{B})\left(P_{(A\cap B)-q}(0)+q\right)\\
	&=(N_{A}+N_{B})\left(P_{A\cap B}(q)\right),
	\end{align*}
	which completes the proof.
\end{proof}

The next result shows that, when $P_A(0)=P_B(0)$, any point in the segment with end points $P_A(0)$ and $(2\beta-1)P_A(0)$ is a fixed point of the mapping $T_{\alpha,\beta}$. Thus, if $P_A(0)=P_B(0)\neq 0$, the mapping $T_{\alpha,\beta}$ has multiple fixed points.

\begin{proposition}\label{prop_especial_cases_FixT}
	Let $A,B\subseteq\Hi$ be nonempty, closed and convex, and let $\alpha,\beta\in\,]0,1]$. The following hold.
	\begin{itemize}
		\item[(i)] If $P_A(0)\in B$, then $(2\beta-1)P_A(0)\in \Fix T_{\alpha,\beta}$.
		\item[(ii)] If $P_B(0)\in A$, then $P_B(0)\in \Fix T_{\alpha,\beta}$.
		\item[(iii)] If $P_A(0)=P_B(0)$, then
		$$
		(1-2\lambda(1-\beta))P_A(0)\in \Fix T_{\alpha,\beta} \quad \text{for all } \lambda\in[0,1].
		$$
	\end{itemize}
\end{proposition}
\begin{proof}
	(i) This assertion can be deduced from the second part of Theorem~\ref{th:fixT_cone}. Indeed, if $P_A(0)\in B$, then $P_{A\cap B}(0)=P_A(0)$. Since $-P_A(0)\in N_A(P_A(0))$, we may take $d_A:=-P_A(0)$ and $d_B:=0$ and~\eqref{eq:prop_fixT_cone_aux1} holds. Thus, taking $x$ as in~\eqref{eq:x_dA_dB}, we have
	$$x=P_A(0)-2(1-\beta)P_A(0)=(2\beta-1)P_A(0),$$
	which is a fixed point of $T_{\alpha,\beta}$ by~\eqref{eq:prop_fixT_cone_aux6}.
	
	(ii) Analogous to the previous one.	
	
	(iii)	Use (i) and (ii) together with Fact~\ref{fact:convex_closed_FixT}.
\end{proof}

Next, we show some results regarding the range of the operator $I-T_{A,B,\alpha,\beta}$, which will be useful later having in mind Fact~\ref{fact:pazy}.

\begin{lemma}\label{lem:vector_v}
	Let $A,B\subseteq\Hi$ be nonempty, closed and convex sets, and let $\alpha,\beta\in\,]0,1[$. Then
	\begin{itemize}
		\item[(i)] $x-T_{A,B,\alpha,\beta}(x)=2\alpha\beta \left(P_A(x)-P_B(2\beta P_A(x)-x)\right), \quad \forall x\in\Hi$;
		\item[(ii)] $\ran (I-T_{A,B,\alpha,\beta})=\ran (I-T_{A+q,B+(2\beta-1)q,\alpha,\beta})+4\alpha\beta(\beta-1)q,\quad\forall q\in\Hi.$
	\end{itemize}
\end{lemma}
\begin{proof} Assertion (i) is straightforward from the definition of $T_{A,B,\alpha,\beta}$: for any $x\in\Hi$, we have
	\begin{align*}
	x-T_{A,B,\alpha,\beta}(x) & =  x-(1-\alpha)x-\alpha(2\beta P_B-I)(2\beta P_A - I)(x)\\
	& = \alpha\left( x - (2\beta P_B(2\beta P_A(x) - x)+2\beta P_A(x)-x \right)\\
	& = 2\alpha\beta\left( P_A(x)-P_B(2\beta P_A(x)-x) \right).
	\end{align*}
	
	To prove (ii), pick any $x,q\in\Hi$. By using the translation formula for projections given in Fact~\ref{fact:projection}(iv), we obtain
	\begin{align*}
	P_A(x)&-P_B(2\beta P_A(x)-x)	 \\
	&= P_{A+q}(x+q)-q-P_B(2\beta P_{A+q}(x+q)-2\beta q-x)\\
	& = P_{A+q}(x+q)-q-P_{B+(2\beta-1)q}(2\beta P_{A+q}(x+q)-x-q)+(2\beta-1)q\\
	& = P_{A+q}(x+q)-P_{B+(2\beta-1)q}\left(2\beta P_{A+q}(x+q)-(x+q)\right)+2(\beta-1)q.	
	\end{align*}
	Therefore, by assertion (i), we get
	$$
	\left(I-T_{A,B,\alpha,\beta}\right)(x)= \left(I-T_{A+q,B+(2\beta-1)q,\alpha,\beta}\right)(x+q)+4\alpha\beta(\beta-1)q, \quad \forall x\in\Hi,
	$$
	and we are done.
\end{proof}

\begin{theorem}\label{th:vector_v}
	Let $A,B\subseteq\Hi$ be nonempty, closed and convex sets, and let $\alpha,\beta\in\,]0,1[$. Suppose that one of the following holds:
	\begin{itemize}
		\item[(i)] $\Hi$ is finite-dimensional.
		\item[(ii)] $\inte A\neq\emptyset$ or $\inte B\neq\emptyset$;
	\end{itemize}
	Then the unique element of minimum norm in $\overline{\ran}\left(I-T_{\alpha,\beta}\right)$ is $2\alpha\beta v$, where $v=P_{\overline{A-B}}(0)$.
\end{theorem}
\begin{proof}
	Let $w$ be the unique element of minimum norm in $\overline{\ran}(I-T_{\alpha,\beta})$. By Lemma~\ref{lem:vector_v}(i), we have ${\ran}(I-T_{\alpha,\beta})\subseteq 2\alpha\beta(A-B)$, which implies $w\in (2\alpha\beta)\overline{A-B}$.
	
	Suppose that (i) holds. Pick any $a\in\rint A$, $b\in\rint B$ and set $q_{a,b}:=\frac{a-b}{2(\beta-1)}$.
	Then, by Lemma~\ref{lem:vector_v}(ii), we have
	\begin{equation}\label{eq:range}
	\ran \left(I-T_{A,B,\alpha,\beta}\right)=\ran\left(I-T_{A+q_{a,b},B+(2\beta-1)q_{a,b},\alpha,\beta}\right)+2\alpha\beta(a-b),
	\end{equation}
	with
	\begin{align*}
	 b+(2\beta-1)q_{a,b}&=b+\frac{2\beta-1}{2(\beta-1)}(a-b)=b+\left(1+\frac{1}{2(\beta-1)} \right)(a-b)\\
	&=a+\frac{a-b}{2(\beta-1)}=a+q_{a,b};
	\end{align*}
	i.e., we have $a+q_{a,b}=b+(2\beta-1)q_{a,b}\in \rint(A+q_{a,b})\cap \rint(B+(2\beta-1)q_{a,b})$.
	Hence, according to Fact~\ref{fact:normal_cone_intersection_formula}(ii) and Theorem~\ref{th:fixT_cone}, the mapping $T_{A+q_{a,b},B+(2\beta-1)q_{a,b},\alpha,\beta}$ has a fixed point and therefore the unique element of minimum norm in $\ran \big(I-T_{A+q_{a,b},B+(2\beta-1)q_{a,b},\alpha,\beta}\big)$ is $0$. Hence, by~\eqref{eq:range}, we deduce
	\begin{equation}\label{eq:minimum_norm}
	\begin{aligned}
	\|w\| &=\inf\left\{\|z\|:z\in\overline{\ran} \left(I-T_{A+q_{a,b},B+(2\beta-1)q_{a,b},\alpha,\beta}\right)+2\alpha\beta(a-b)\right\}\\
	&\leq 2\alpha\beta\|a-b\|+\inf\left\{\|z\|:z\in\overline{\ran} \left(I-T_{A+q_{a,b},B+(2\beta-1)q_{a,b},\alpha,\beta}\right)\right\}\\
	&=2\alpha\beta\|a-b\|,
	\end{aligned}
	\end{equation}
	and this holds for every $a\in\rint A$ and every $b\in\rint B$.
	
	Now, choose any $a\in A, b\in B$. Then, there exist two sequences $\{a_n\}\subset \rint A, \{b_n\}\subset \rint B$ such that $a_n\rightarrow a$ and $b_n\rightarrow b$, and by~\eqref{eq:minimum_norm}, we get
	$$
	2\alpha\beta\|a-b\|=2\alpha\beta\left\|\lim_{n\rightarrow\infty} (a_n-b_n)\right\|=2\alpha\beta\lim_{n\rightarrow\infty}\| a_n-b_n\|\geq \|w\|.
	$$
	Thus, since $ (2\alpha\beta)^{-1}w\in\overline{A-B}$ and $\|(2\alpha\beta)^{-1}w\| \leq \|a-b\|,$ for all $a\in A$ and $b\in B$, it must be that $(2\alpha\beta)^{-1}w= P_{\overline{A-B}}(0)$, which proves (i).
	
	To prove the result when (ii) holds, if $\inte A\neq\emptyset$ ($\inte B\neq\emptyset$), then take any $a\in \inte A$ and $b\in B$ ($a\in A$ and $b\in\inte B$) and repeat the proof of the previous case using Fact~\ref{fact:normal_cone_intersection_formula}(i) instead of Fact~\ref{fact:normal_cone_intersection_formula}(ii).
\end{proof}

We conclude this section by presenting some translation formulas for the AAMR operators in the special case when both sets are closed affine subspaces.

\begin{proposition}\label{prop:affine_const}
	Let $A,B\subseteq\Hi$ be closed affine subspaces with nonempty intersection. Let $y\in A\cap B$ and let $\alpha,\beta\in\,]0,1[$. Then, for any $x\in\Hi$,
	\begin{equation}\label{eq:prop_affine_traslation}
	T_{A,B,\alpha,\beta}(x)=T_{A-y,B-y,\alpha,\beta}(x)+T_{A,B,\alpha,\beta}(0),
	\end{equation}
	and
	\begin{equation}\label{eq:prop_affine_traslation_FixT}
	T_{A,B,\alpha,\beta}(x)=T_{A-y,B-y,\alpha,\beta}(x-z^*)+z^*, \quad \forall z^*\in\Fix T_{A,B,\alpha,\beta}.
	\end{equation}
	Furthermore, one has
	\begin{equation}\label{eq:prop_affine_traslation_FixT2}
	\Fix T_{A,B,\alpha,\beta}=z^*+\Fix T_{A-y,B-y,\alpha,\beta}, \quad \forall z^*\in\Fix T_{A,B,\alpha,\beta}.
	\end{equation}
\end{proposition}
\begin{proof}
	Because $A-y$ and $B-y$ are closed linear subspaces of $\Hi$, then $P_{A-y}$ and $P_{B-y}$ are linear mappings (see Fact~\ref{fact:projection_firmly_nonexapnsive}). Denote the \emph{modified reflector operator} onto any set $C\subset\Hi$ by $Q_{\beta,C}:=2\beta P_C-I$. Then, the mappings $Q_{\beta,A-y}$ and $Q_{\beta,B-y}$ are also linear. Further, for any $z\in\Hi$, by Fact~\ref{fact:projection}(iv), we have
	\begin{equation*}\label{eq:prop_affine_traslation_aux1}
	\begin{aligned}
	Q_{\beta,A}(z) & =2\beta P_A(z)-z = 2\beta P_{A-y}(z-y)+2\beta y-z \\
	&= 2\beta P_{A-y}(z)-z+2\beta (P_{A-y}(-y)+y)\\
	&= Q_{\beta,A-y}(z)+2\beta P_{A}(0)
	=Q_{\beta,A-y}(z)+Q_{\beta,A}(0).
	\end{aligned}
	\end{equation*}
	Similarly, we get $Q_{\beta,B}(z)=Q_{\beta,B-y}(z)+Q_{\beta,B}(0).$
	Combining these equalities together and
	using the linearity of $Q_{\beta,B-y}$, we obtain
	\begin{equation*}
	\begin{aligned}
	Q_{\beta,B}Q_{\beta,A}(x) & =Q_{\beta,B}\left( Q_{\beta,A-y}(x)+Q_{\beta,A}(0) \right)\\
	&=Q_{\beta,B-y} Q_{\beta,A-y}(x)+Q_{\beta,B-y} Q_{\beta,A}(0)+Q_{\beta,B}(0)\\
	&=Q_{\beta,B-y} Q_{\beta,A-y}(x)+Q_{\beta,B} Q_{\beta,A}(0),
	\end{aligned}
	\end{equation*}
	which implies~\eqref{eq:prop_affine_traslation}.
	
	Now take $z^*\in\Fix T_{A,B,\alpha,\beta}$. Then, by~\eqref{eq:prop_affine_traslation}, we have
	\begin{equation*}
	 z^*=T_{A,B,\alpha,\beta}(z^*)=T_{A-y,B-y,\alpha,\beta}(z^*)+T_{A,B,\alpha,\beta}(0).
	\end{equation*}
	By replacing $T_{A,B,\alpha,\beta}(0)=-T_{A-y,B-y,\alpha,\beta}(z^*)+z^*$ in~\eqref{eq:prop_affine_traslation} and using the linearity of the operator $T_{A-y,B-y,\alpha,\beta}$, we obtain~\eqref{eq:prop_affine_traslation_FixT}.
	
	The last assertion easily follows from~\eqref{eq:prop_affine_traslation_FixT}. Indeed, for any $z^*\in\Fix T_{A,B,\alpha,\beta}$, one has
	\begin{align*}
	w^\star\in \Fix T_{A,B,\alpha,\beta}&\Leftrightarrow T_{A,B,\alpha,\beta}(w^\star)=w^\star\Leftrightarrow T_{A-y,B-y,\alpha,\beta}(w^\star-z^\star)+z^\star=w^\star\\
	&\Leftrightarrow w^\star-z^\star\in\Fix T_{A-y,B-y,\alpha,\beta},
	\end{align*}
	which implies~\eqref{eq:prop_affine_traslation_FixT2}.
\end{proof}

\section{New projection scheme for finding the closest point in the intersection}\label{sec:method}

In the main result of this section we show that the iterative methods defined by the AAMR operators in~\eqref{eq:general_case_sequence} are weakly convergent to a fixed point of the operators, and the \emph{shadow sequences} ${\left(P_A(q+x_n)\right)}_{n=0}^\infty$ are strongly convergent to the solution to problem~\eqref{eq:general_case_problem}.

\begin{theorem}\label{theorem:NPM_convergence}
	Let $A,B\subseteq\Hi$ be nonempty closed and convex sets. Fix any $\alpha,\beta\in\,]0,1[$. Given $q\in\Hi$, choose any  $x_0\in\Hi$ and consider the sequence defined by
	\begin{equation}\label{eq:NPM_sequence_iteration}
	x_{n+1}=T_{A-q,B-q,\alpha,\beta}(x_n), \quad n=0,1,2\ldots
	\end{equation}
	Then, if  $A\cap B\neq\emptyset$ and $q-P_{A\cap B}(q)\in (N_A+N_B)\left(P_{A\cap B}(q)\right)$, the following assertions hold
	\begin{enumerate}
	\item[(i)] the sequence ${(x_n)}_{n=0}^\infty$ is weakly convergent to a point $x^\star\in \Fix T_{A-q,B-q,\alpha,\beta}$ such that
	\begin{equation}\label{eq:x_star}
	P_A(q+x^\star)=P_{A\cap B}(q);
	\end{equation}
	\item[(ii)] the sequence ${(x_{n+1}-x_n)}_{n=0}^\infty$ is strongly convergent to 0;
	\item[(iii)] the sequence ${\left(P_A(q+x_n)\right)}_{n=0}^\infty$ is strongly convergent to $P_{A\cap B}(q)$.
	\end{enumerate}	
	Otherwise, $\|x_n\|\to\infty$.

Furthermore, if both $A$ and $B$ are closed affine subspaces with nonempty intersection and $q-P_{A\cap B}(q)\in(A-A)^\bot+(B-B)^\bot$, then the sequence ${(x_n)}_{n=0}^\infty$ is strongly convergent to $P_{\Fix T_{A-q,B-q,\alpha,\beta}}(x_0)$.
\end{theorem}
\begin{proof}
	By Corollary~\ref{cor:NPM_fixT_cone}, we know that $\Fix T_{A-q,B-q,\alpha,\beta}\neq\emptyset$. The projector operators $P_{A-q}$ and $P_{B-q}$ are firmly nonexpansive by Fact~\ref{fact:projection_firmly_nonexapnsive}. Then, as $\beta\in\,]0,1[$, the operators $T_1:=\beta P_{A-q}$ and $T_2:=\beta P_{B-q}$ are also firmly nonexpansive, and moreover $\frac{1}{\beta}$-cocoercive (with $\frac{1}{\beta}>1$). Observe that $T_{A-q,B-q,\alpha,\beta}=(1-\alpha)I+\alpha(2T_2-I)(2T_1-I)$,
	with $\Fix((2T_2-I)(2T_1-I))=\Fix T_{A-q,B-q,\alpha,\beta}\neq\emptyset$.
	Hence, we can use Fact~\ref{fact:convergence} with $\lambda_n:=\alpha$ to show that the operator $T_{A-q,B-q,\alpha,\beta}$ has a fixed point $x^\star$ such that the sequence defined by~\eqref{eq:NPM_sequence_iteration} satisfies
	$$
	x_n \rightharpoonup  x^\star\in \Fix T_{A-q,B-q,\alpha,\beta},\quad x_{n+1}-x_n\to 0\quad \text{and} \quad \beta P_{A-q}(x_n) \rightarrow \beta P_{A-q}(x^\star).
	$$
	Moreover, by Proposition~\ref{prop:projected_fixT}, we have
	$$P_{A-q}(x^\star)=P_{(A-q)\cap (B-q)}(0),$$
	which, by Fact~\ref{fact:projection}(iv), is equivalent to
	\begin{equation*}
	P_A(q+x^\star)=P_{A\cap B}(q),
	\end{equation*}
	and thus ${P_A(q+x_n)} \rightarrow P_{A\cap B}(q)$. This concludes the proof of statements (i)--(iii).

	The case where $\Fix T_{A-q,B-q,\alpha,\beta}=\emptyset$ easily follows from Corollary~\ref{cor:NPM_fixT_cone} and Fact~\ref{fact:pazy}, since $T_{A-q,B-q,\alpha,\beta}$ is $\alpha$-averaged, according to Proposition~\ref{prop:nonexpansiveness_gdr}.	

Finally, assume that both $A$ and $B$ are closed affine subspaces. By Fact~\ref{fact:subspaces_normalcone}(i), we have
	$$
	N_A(x)=N_{A-x}(0)=(A-x)^\bot=(A-A)^\bot, \quad \text{for all }  x\in A.
	$$
	Hence, $N_A(x)=(A-A)^\bot$ for all $x\in A$, and likewise, $N_B(x)=(B-B)^\bot$ for all $x\in B$. Therefore, we have $q-P_{A\cap B}(q)\in(N_A+N_B)(P_{A\cap B}(q))$, which implies that $\Fix T_{A-q,B-q,\alpha,\beta}\neq\emptyset$. Thus, taking $y\in A\cap B$ and $z^*\in	 \Fix T_{A-q,B-q,\alpha,\beta}$, since $y-q\in(A-q)\cap(B-q)$, we can apply Proposition~\ref{prop:affine_const} recursively to get
	$$
	x_n=T^n_{A-q,B-q,\alpha,\beta}(x_0)=T^n_{A-y,B-y,\alpha,\beta}(x_0-z^*)+z^*.
	$$
By assertion (ii), we know that
$$T_{A-y,B-y,\alpha,\beta}^{n+1}(x_0-z^\star)-T_{A-y,B-y,\alpha,\beta}^{n}(x_0-z^\star)=x_{n+1}-x_n\rightarrow 0.$$
Since $T_{A-y,B-y,\alpha,\beta}$ is linear, by Fact~\ref{fact:linearmap_strong_convergence}, we deduce
	$$
	T^n_{A-y,B-y,\alpha,\beta}(x_0-z^*)\rightarrow P_{\Fix T_{A-y,B-y,\alpha,\beta}}(x_0-z^*).
	$$
	Consequently,
	$$
	x_n=T^n_{A-q,B-q,\alpha,\beta}(x_0)\rightarrow P_{\Fix T_{A-y,B-y,\alpha,\beta}}(x_0-z^*)+z^*=P_{\Fix T_{A-q,B-q,\alpha,\beta}}(x_0),
	$$
	where the last equality holds by Fact~\ref{fact:projection}(iv) and Proposition~\ref{prop:affine_const}.
\end{proof}

\begin{remark}
	If any of the conditions given in Fact~\ref{fact:normal_cone_intersection_formula} or in Fact~\ref{fact:chip_subspaces} hold, then the \emph{constraint qualification} $q-P_{A\cap B}(q)\in (N_A+N_B)\left(P_{A\cap B}(q)\right)$ holds at every point $q\in\Hi$, and the convergence of the sequence defined by~\eqref{eq:NPM_sequence_iteration} is guaranteed.
\end{remark}

\begin{remark} \emph{(i)} Observe that Theorem~\ref{theorem:NPM_convergence}(iii) still holds for $\alpha=1$, since Corollary~\ref{cor:NPM_fixT_cone} remains valid for $\alpha=1$ and $\lambda_n:=1$ satisfies the hypothesis of Fact~\ref{fact:convergence}(b).\\
	\emph{(ii)} It is straightforward to construct a version of the AAMR algorithm where the value of $\alpha$ may vary across the iterations. Specifically, Theorem~\ref{theorem:NPM_convergence}(i)--(iii) still holds if one replaces~\eqref{eq:NPM_sequence_iteration} by the iterative method
$$x_{n+1}=T_{A-q,B-q,\alpha_n,\beta}(x_n)=(1-\alpha_n)x_n+\alpha_n(2\beta P_B-I)(2\beta P_A-I)(x_n), \quad n=0,1,2\ldots$$
where $\{\alpha_n\}_{n\geq 0}\subset\,]0,1[$ satisfies  $\sum_{n=0}^\infty{\alpha_n(1-\alpha_n)}=+\infty$ and $\inf_{n\geq 0}\alpha_n>0$.\\
\emph{(iii)} Similarly, it is easy to include errors in the AAMR scheme as in~\cite{C09}, by using in the proof of Theorem~\ref{theorem:NPM_convergence} a version of Fact~\ref{fact:convergence} with errors (see~\cite[Theorem~2.1]{C09}).
\end{remark}

In~\cite[Theorem~3.2]{DLW97} and~\cite[Theorem~10.13]{D01}, the existence of a point $x^\star$ satisfying~\eqref{eq:x_star} for every $q\in \mathcal{H}$ is proved to be equivalent to the strong CHIP, for the particular case where $B=L^{-1}(b)$, for some bounded linear operator $L$ from $\Hi$ into a finite-dimensional Hilbert space $Y$ and $b\in Y$. In addition, Deutsch an Ward propose in~\cite{DLW97} a steepest descent method with line search for finding the point  $x^\star$, whose linear convergence is proved under the additional assumption that $C$ is polyhedral.

If the sets $A$ and $B$, with nonempty intersection, have the strong CHIP at the point $P_{A\cap B}(q)$, then trivially $q-P_{A\cap B}(q)\in (N_A+N_B)\left(P_{A\cap B}(q)\right)$ and the scheme converges. However, in the following example we show that the strong CHIP is not a necessary condition for the AAMR method to converge for a particular point $q\in\Hi$.

\begin{example}\label{ex:counterex_strong_CHIP} Consider the setting $\Hi=\mathbb{R}^2$, $A:=(1,1)+\mathbb{B}$ and $B:=(-1,1)+\mathbb{B}$. The pair of sets $\{A,B\}$ does not have strong CHIP at any point, since $A\cap B=\{(0,1)\}$ and
	$$
	(N_A+N_B)((0,1))=\mathbb{R}\times\{0\}\neq\mathbb{R}^2=	N_{A\cap B}((0,1)).
	$$
	If we take any $q\in\mathbb{R}\times\{1\}$, then $$q-P_{A\cap B}(q)=q-(0,1)\in\mathbb{R}\times\{0\}=(N_A+N_B)(P_{A\cap B}(q)),$$
	and by Theorem~\ref{theorem:NPM_convergence}, the AAMR method defined by~\eqref{eq:NPM_sequence_iteration} will generate a  sequence that converges to a point $x^\star$ such that $P_A(x^\star+q)=\{(0,1)\}$.
	On the other hand, if $q\not\in\mathbb{R}\times\{1\}$, then $q-P_{A\cap B}(q)\not\in(N_A+N_B)(P_{A\cap B}(q))$, and the sequence ${(x_n)}_{n=0}^\infty$ generated by~\eqref{eq:NPM_sequence_iteration} will not converge, having $\|x_n\|\to\infty$.
\end{example}

We have seen that even when the strong CHIP does not hold, the method can converge for a particular point $q\in\Hi$. However, the following result says that if we want the method to converge for every point in $\Hi$, the strong CHIP will have to be required.

\begin{proposition}\label{prop:strong_CHIP}
	Let $A,B\subseteq\Hi$ be nonempty, closed and convex subsets with nonempty intersection. Then the following are equivalent:
	\begin{itemize}
		\item[(i)] $\{A,B\}$ has the strong CHIP;
		\item[(ii)] for all $q\in\Hi$,
		$$
		q-P_{A\cap B}(q)\in (N_A+N_B)(P_{A\cap B}(q)).
		$$		
	\end{itemize}
\end{proposition}
\begin{proof}
	Assume~(ii). Take $x\in A\cap B$ and let $y\in N_{A\cap B}(x)$. By Fact~\ref{fact:projection_normalcone}, $P_{A\cap B}(x+y)=x$. Hence, by~(ii), we have
	$$
	y=x+y-P_{A\cap B}(x+y)\in(N_A+N_B)(P_{A\cap B}(x+y)).
	$$
	Therefore $N_{A\cap B}(x)\subseteq (N_A+N_B)(x)$. Since $x$ is arbitrary in $A\cap B$ and the reverse inclusion always holds, then $\{A,B\}$ has the strong CHIP, which proves that (ii)$\Rightarrow$(i). The opposite implication clearly holds by Fact~\ref{fact:projection_normalcone}.
\end{proof}

We finish this section with the following consequence of Theorem~\ref{th:vector_v}, which holds even when $A\cap B=\emptyset$.
\begin{corollary}
	Let $A,B\subseteq\Hi$ nonempty closed and convex, let $\alpha,\beta\in\,]0,1[$ and let $x_0\in\Hi$. For any $q\in\Hi$, consider the iterated sequence defined by
	\begin{equation*}
	x_{n+1}=T_{A-q,B-q,\alpha,\beta}(x_n),\quad n=0,1,2,\dots
	\end{equation*}
	Suppose that one of the following holds:
	\begin{itemize}
		\item[(i)] $\Hi$ is finite-dimensonal;
		\item[(ii)] $\inte A\neq\emptyset$ or $\inte B\neq\emptyset$.
	\end{itemize}
	Then, the sequence $(x_n-x_{n+1})_{n=0}^\infty$ converges in norm to $2\alpha\beta P_{\overline{A-B}}(0)$.	
\end{corollary}
\begin{proof}
	Use Fact~\ref{fact:pazy} together with Theorem~\ref{th:vector_v}.	
\end{proof}

\section{Finitely many sets}\label{sec:several_sets}

In this section we show how to apply the AAMR method to the case of finitely many sets. Let $C_1,\ldots,C_r\subset\Hi$ be nonempty, closed and convex subsets of $\Hi$. Given any $q\in\Hi$, we are interested in solving  the problem
\begin{equation}\label{eq:FMS_general_case_problem}
\text{Find } p=P_{\bigcap_{i=1}^r C_i}(q).
\end{equation}

To solve problem~\eqref{eq:FMS_general_case_problem} with the AAMR method, we use the following well-known Pierra's \emph{product-space reformulation}~\cite{Pierra}. Consider the product space $\Hi^r$ and define the sets
\begin{equation*}
C:=\prod_{i=1}^{r}C_i, \quad D:=\{(x,x,\ldots,x)\in\Hi^r:x\in\Hi\}.
\end{equation*}
While the set $D$, sometimes called the \emph{diagonal}, is always a closed subspace, the properties of $C$ are largely inherited. For instance, $C$ is nonempty, closed and convex.
Since
\begin{equation*}
p\in\bigcap_{i=1}^r C_i \Leftrightarrow (p,p,\ldots,p)\in C\cap D,
\end{equation*}
by Fact~\ref{fact:projection}(ii), we have the following equivalent reformulation of problem~\eqref{eq:FMS_general_case_problem}:
$$\text{Find $p$ such that } (p,p,\ldots,p)=P_{C\cap D}(q,q,\ldots,q).$$
Moreover, knowing the projections onto $C_1,\ldots,C_r$, the projections onto $C$ and $D$ can be easily computed.
Indeed, for any $\mathbf{x}=(x_1,\ldots,x_r)\in\Hi^r$, we have
\begin{equation*}
P_C(\mathbf{x})=\left( P_{C_1}(x_1),\ldots,P_{C_r}(x_r)\right),
\end{equation*}
and,
\begin{equation*}
P_D(\mathbf{x})=\left( \frac{1}{r}\sum_{i=1}^r x_i, \ldots, \frac{1}{r}\sum_{i=1}^r x_i \right),
\end{equation*}
see~\cite[Lemma~1.1]{Pierra}. For further details see, for example,~\cite[Section~3]{ABTcomb}.

Therefore, the AAMR operators
$$T_{D,C,\alpha,\beta}=(1-\alpha)I+\alpha(2\beta P_C-I)(2\beta P_D-I)$$ can be readily computed whenever $P_{C_1},P_{C_2},\ldots,P_{C_r}$ can be.

To derive our main result regarding the convergence of the AAMR method for finitely many sets, we will use the following characterization to rewrite the constraint qualification in the product space.

\begin{lemma}\label{lem:FMS_normal_cones}
	For every $\mathbf{x}=(x,x,\ldots,x)\in C\cap D$,
	\begin{equation*}
	\left( N_C(\mathbf{x})+N_D(\mathbf{x}) \right)\cap D=\left( \sum_{i=1}^r N_{C_i}(x)\right)^r\cap D.
	\end{equation*}
\end{lemma}
\begin{proof}
	Let $\mathbf{x}=(x,x,\ldots,x)\in C\cap D$.  Since $C$ is the product of $r$ sets and $D$ is a closed subspace of $\Hi^r$, we have
	\begin{align*}
	N_C(\mathbf{x})&=N_{C_1}(x)\times N_{C_2}(x) \times \cdots \times N_{C_r}(x),\\
	N_D(\mathbf{x})&=D^\bot=\left\{ \mathbf{u}\in\Hi^r: \sum_{i=1}^r u_i=0 \right\}.
	\end{align*}
	
	To prove the direct inclusion, pick any $\mathbf{y}=(y,y,\ldots,y)\in \left( N_C(\mathbf{x})+N_D(\mathbf{x}) \right)\cap D$. Then,
	\begin{equation*}
	y\in N_{C_i}(x)+u_i, \quad \text{for } i=1,\ldots,r,
	\end{equation*}
	with $\sum_{i=1}^r u_i=0$. Thus, as $N_{C_i}(x)$ are all cones, we have $y\in\sum_{i=1}^r N_{C_i}(x)$, which yields $\mathbf{y}\in\left( \sum_{i=1}^r N_{C_i}(x)\right)^r$.
	
	To prove the reverse inclusion, pick any $\mathbf{y}=(y,y,\ldots,y)\in \left( \sum_{i=1}^r N_{C_i}(x)\right)^r$. Then, there exists $d_i\in N_{C_i}(x)$, for each $i=1,\ldots, r$, such that $y=\sum_{i=1}^r d_i$. Let $\mathbf{d}:=(d_1,d_2,\ldots,d_r)\in N_C(\mathbf{x})$ and $\mathbf{u}:=(u_1,u_2,\ldots,u_r)$, where $u_i:=y-rd_i$, for each $i=1,\ldots, r$. Since
	$$
	\sum_{i=1}^r u_i=\sum_{i=1}^r(y-rd_i)=ry-r	\sum_{i=1}^n d_i=0,
	$$
	we have $\mathbf{u}\in D^\bot$.
	Hence, $\mathbf{y}=r\mathbf{d}+\mathbf{u}\in N_C(\mathbf{x})+N_D(\mathbf{x})$.
\end{proof}

We are now ready to derive the following theorem of convergence of the AAMR method for the case of finitely many sets.

\begin{theorem}\label{th:FMS}
	Let $C_1,C_2,\ldots,C_r\subseteq\Hi$ be nonempty closed and convex sets. Fix any $\alpha,\beta\in\,]0,1[$. Given $q\in\Hi$, choose any $\mathbf{x}_0\in\Hi^r$  and consider the sequence $(\mathbf{x}_n)_{n=0}^\infty={(x_{n,1},\ldots,x_{n,r})}_{n=0}^\infty$,  defined by
	\begin{equation}\label{eq:FMS_sequence_iteration}
	\mathbf{x}_{n+1}=T_{D,\prod_{i=1}^r (C_i-q),\alpha,\beta}(\mathbf{x}_n),\quad n=0,1,2,\ldots
	\end{equation}
	Then, if $\bigcap_{i=1}^r C_i\neq \emptyset$ and $q-P_{\bigcap_{i=1}^r C_i}(q)\in \sum_{i=1}^r N_{C_i}\left(P_{\bigcap_{i=1}^r C_i}(q)\right)$, the following assertions hold:
	\begin{enumerate}
	\item[(i)] the sequence ${(\mathbf{x}_n)}_{n=0}^\infty$ is weakly convergent to a point $\mathbf{x}^\star=(x_1^\star,\ldots,x_r^\star)\in \Fix T_{D,\prod_{i=1}^r (C_i-q),\alpha,\beta}$ such that
	$$
	P_{\bigcap_{i=1}^r C_i}(q)=q+\frac{1}{r}\sum_{j=1}^r x_j^\star;
	$$
	\item[(ii)] the sequence $\left( q+\frac{1}{r}\sum_{j=1}^r x_{n,j}\right)_{n=0}^{\infty}$ is strongly convergent to $P_{\bigcap_{i=1}^r C_i}(q)$.
	\end{enumerate}
	Otherwise, $\|\mathbf{x}_n\|\to\infty$.

Furthermore, if $C_1,C_2,\ldots,C_r$ are closed affine subspaces with nonempty intersection satisfying $q-P_{\bigcap_{i=1}^r C_i}(q)\in \sum_{i=1}^r (C_i-C_i)^\perp$, then the sequence ${(\mathbf{x}_n)}_{n=0}^\infty$ is strongly convergent to  $P_{\Fix T_{D,\prod_{i=1}^r (C_i-q),\alpha,\beta}}(\mathbf{x}_0)$.
\end{theorem}
\begin{proof}
	Let $\mathbf{q}:=(q,q,\ldots,q)\in\Hi^r$. Observe that $C-\mathbf{q}=\prod_{i=1}^r (C_i-q)$ and $D-\mathbf{q}=D$, since $D$ is a subspace containing $\mathbf{q}$. Therefore, the operator defining the iteration~\eqref{eq:FMS_sequence_iteration} is simply $T_{D-q,C-q,\alpha,\beta}$. Observe also that  $\bigcap_{i=1}^r C_i\neq\emptyset$ if and only if $C\cap D\neq\emptyset$.
	
	Let $p:=P_{\bigcap_{i=1}^r C_i}(q)$. Then, $\mathbf{p}:=(p,p,\ldots,p)=P_{C\cap D}(\mathbf{q})$. Moreover, by Lemma~\ref{lem:FMS_normal_cones}, we have
	$$
	\mathbf{q}-\mathbf{p}\in N_C(\mathbf{p})+N_D(\mathbf{p}) \Leftrightarrow  q-p\in\sum_{i=1}^r N_{C_i}(p).
	$$
	The result thus follows from Theorem~\ref{theorem:NPM_convergence}.
\end{proof}

\begin{remark} \emph{(i)} The order of action of the projections onto $D$ and $C$ chosen in Theorem~\ref{th:FMS} makes the shadow sequence $P_D(\mathbf{x}_{n}+\mathbf{q})$ to lay in the diagonal. In this way, it can be identified with a sequence in the original space to be monitored; concretely, the sequence $$\left( q+\frac{1}{r}\sum_{j=1}^r x_{n,j}\right)_{n=0}^{\infty}\rightarrow P_{\bigcap_{i=1}^r C_i}(q).$$
	\emph{(ii)} Thanks to Lemma~\ref{lem:FMS_normal_cones}, it is straightforward to prove an analogous result to Proposition~\ref{prop:strong_CHIP}, showing that strong CHIP of $\{C_1,\ldots, C_n\}$ characterizes the weak convergence of the iterative method~\eqref{eq:FMS_sequence_iteration} for every point $q\in\Hi$.
\end{remark}

Let us now show some similarities (and differences) between AAMR and a method introduced by Combettes in~\cite{C09}. In this work, the author proposed a strongly convergent algorithm for computing the resolvent of a finite sum of maximally monotone operators. When these operators are chosen as the normal cones of closed and convex sets and the strong CHIP holds, the resolvent of the sum of the operators is nothing but the projection onto the intersection of the sets. The algorithm introduced in~\cite[Theorem~2.8]{C09} also relies on the product space $\Hi^r$ and is defined by the recurrence
\begin{align}\label{eq:combettes}
\mathbf{z}_{n+1}&=\left(1-\frac{\lambda_n}{2}\right)\mathbf{z}_n+\frac{\lambda_n}{2}R_{{D}}\left( 2P_{{C}}\left(\frac{\mathbf{z}_n+\gamma\mathbf{q}}{\gamma+1}\right)- \mathbf{z}_n\right),\\
\mathbf{x}_{n+1}&=P_DP_C(\mathbf{z}_{n+1}),\label{eq:combettes_bis}
\end{align}
for $\gamma>0$, and $(\lambda_n)_{n\geq 0}\in\,]0,2]$ such that $\inf_{n\geq 0}\lambda_n>0$. Observe that, by the dilatation formula in Fact~\ref{fact:projection}(v) and the linearity of $R_D$ given by Fact~\ref{fact:projection_firmly_nonexapnsive}, we have
\begin{align*}
R_{{D}}\left( 2P_{{C}}\left(\frac{\mathbf{z}_n+\gamma\mathbf{q}}{\gamma+1}\right)- \mathbf{z}_n\right)&=R_{{D}}\left( 2\frac{1}{\gamma+1}\left(P_{{(\gamma+1)C-\gamma\mathbf{q}}}\left(\mathbf{z}_n\right)+\gamma\mathbf{q}\right)- \mathbf{z}_n\right)\\
&= R_{{D}}\left( 2\frac{1}{\gamma+1}P_{{(\gamma+1)C-\gamma\mathbf{q}}}\left(\mathbf{z}_n\right)- \mathbf{z}_n\right)+R_D\left( 2\frac{\gamma}{\gamma+1}\mathbf{q}\right)\\
&=R_{{D}}\left( 2\frac{1}{\gamma+1}P_{{(\gamma+1)C-\gamma\mathbf{q}}}\left(\mathbf{z}_n\right)- \mathbf{z}_n\right)+ 2\frac{\gamma}{\gamma+1}\mathbf{q}.
\end{align*}
Thus, setting $\beta:=\frac{1}{1+\gamma}$ and $\alpha_n:=\frac{\lambda_n}{2}$, the recurrence in \eqref{eq:combettes} can be expressed as
\begin{equation}\label{eq:combettes2}
\mathbf{z}_{n+1}=\left(1-\alpha_n\right)\mathbf{z}_n+\alpha_n R_{{D}}\left( 2\beta P_{{\frac{1}{\beta}C-\frac{1-\beta}{\beta}\mathbf{q}}}\left(\mathbf{z}_n\right)- \mathbf{z}_n\right)+ 2\alpha_n (1-\beta)\mathbf{q};
\end{equation}
or equivalently, in terms of the AAMR operator~\eqref{eq:def_GDR},
\begin{equation}\label{eq:combettes3}
\mathbf{z}_{n+1}=T_{\prod_{i=1}^r \left(\frac{1}{\beta}C_i-\frac{1-\beta}{\beta}{q}\right),D,\alpha_n,\beta}(\mathbf{z}_n)+ 2(1-\beta)\alpha_n P_{{D}}\left( 2\beta P_{{\frac{1}{\beta}C-\frac{1-\beta}{\beta}\mathbf{q}}}\left(\mathbf{z}_n\right)- \mathbf{z}_n +\mathbf{q}\right),
\end{equation}
with $\beta\in\,]0,1[$ and $\alpha_n\in\,]0,1]$. The latter scheme clearly differs from AAMR, even when $\mathbf{q}=\mathbf{0}$.

\section{Numerical Experiments}\label{sec:numerical}

In this section we show the results of four different numerical experiments with the common setting of finding the projection of an arbitrary point onto the intersection of two closed subspaces $U$ and $V$ in the Euclidean space $\Hi=\R^{50}$ such that $U\cap V\neq \{0\}$. We compare the new AAMR method with Combettes' method (CM) given by~\eqref{eq:combettes_bis}--\eqref{eq:combettes2}, the method of alternating projections (MAP), the Douglas--Rachford method (DRM) and Haugazeau's method in its basic form (see, e.g.,~\cite[equation~(18)]{BCL06}), and we test the influence of the parameter $\beta$ in the behavior of the AAMR method to see which value gives better convergence results. We have also tested the HLWB method with parameters $\lambda_n:=1/(n+1)$ (see~\cite{Ba96}), but we only show the results obtained in Figure~\ref{fig:num_exp_dts}, as the method was clearly outperformed by all the other algorithms in our experiments.

The rate of linear convergence of DRM for subspaces is known to be the cosine of the Friedrichs angle~\cite{BCNPW14,BCNPW15} (see Definition~\ref{def:F_angle}), while the rate of convergence of MAP is the squared cosine, see~\cite{A50,De83,DH97}. It was then compulsory to take the Friedrichs angle into consideration in our numerical experiments. In our tests we computed Friedrichs angles from principal angles, see~\cite{BCNPW15} for further information.

Observe that, for DRM and AAMR, the sequences of interest to be monitored are, respectively,
$$
\left(P_U(DR_{U,V}^n(x_0))\right)_{n=0}^\infty\quad \text{and} \quad \left(P_U(T_{U-x_0,V-x_0,\alpha,\beta}^n(x_0)+x_0)\right)_{n=0}^\infty
$$
as these are the sequences that converge to the desired point $P_{U\cap V}(x_0)$, while for MAP, CM and Haugazeau's method, the sequence $(x_i)_{n=0}^{\infty}$ given by the respective algorithm is directly the sequence of interest.
We used a stopping criterion based on the true error; that is, we terminated the algorithms when the current iterate of the monitored sequence $(z_n)_{n=0}^\infty$ satisfies
$$
d_{U\cap V}(z_n)<\varepsilon
$$
for the first time (in real situations this information is not usually available). As in the numerical experiments in~\cite{BCNPW14}, the tolerance was set to $\varepsilon:=10^{-3}$.

The purpose of our first experiment was to find out which value of $\alpha$ is optimal for AAMR when it is applied to subspaces. To this aim, we randomly generated $1000$ pairs of subspaces and run AAMR with a random starting point for each value of $\alpha\in\{0.01,0.02,\ldots,0.99,1\}$ and $\beta\in\{0.6,0.7,0.8,0.9\}$. In Figure~\ref{fig:best_alpha} we have plotted the best value $\alpha$ against the Friedrichs angle, that is, the value of $\alpha$ for which AAMR was faster. For a fair comparison in the subsequent tests, we performed the same experiment with CM.
\begin{figure}[ht!]
    \centering
	\includegraphics[width=.49\linewidth]{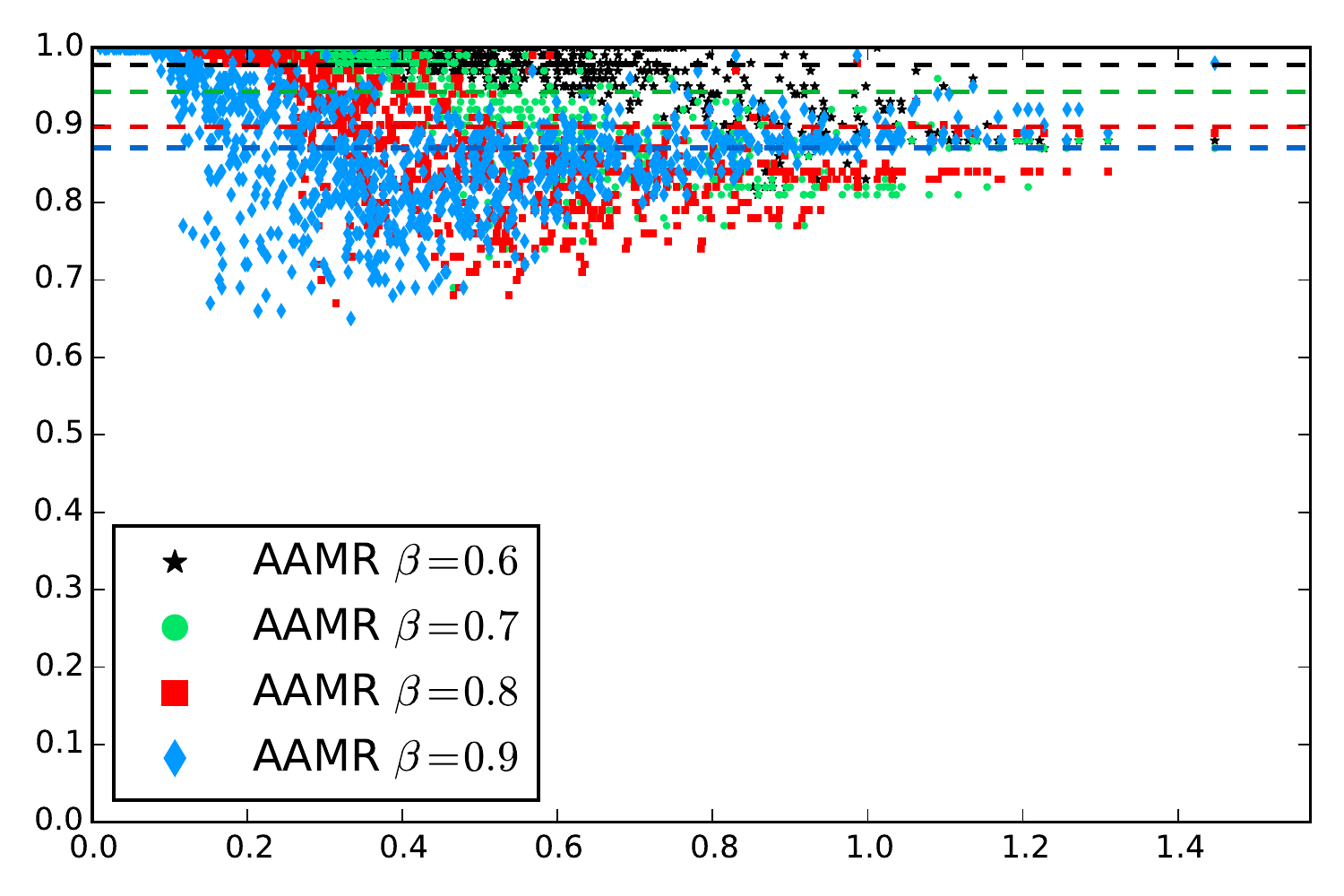}
	\includegraphics[width=.49\linewidth]{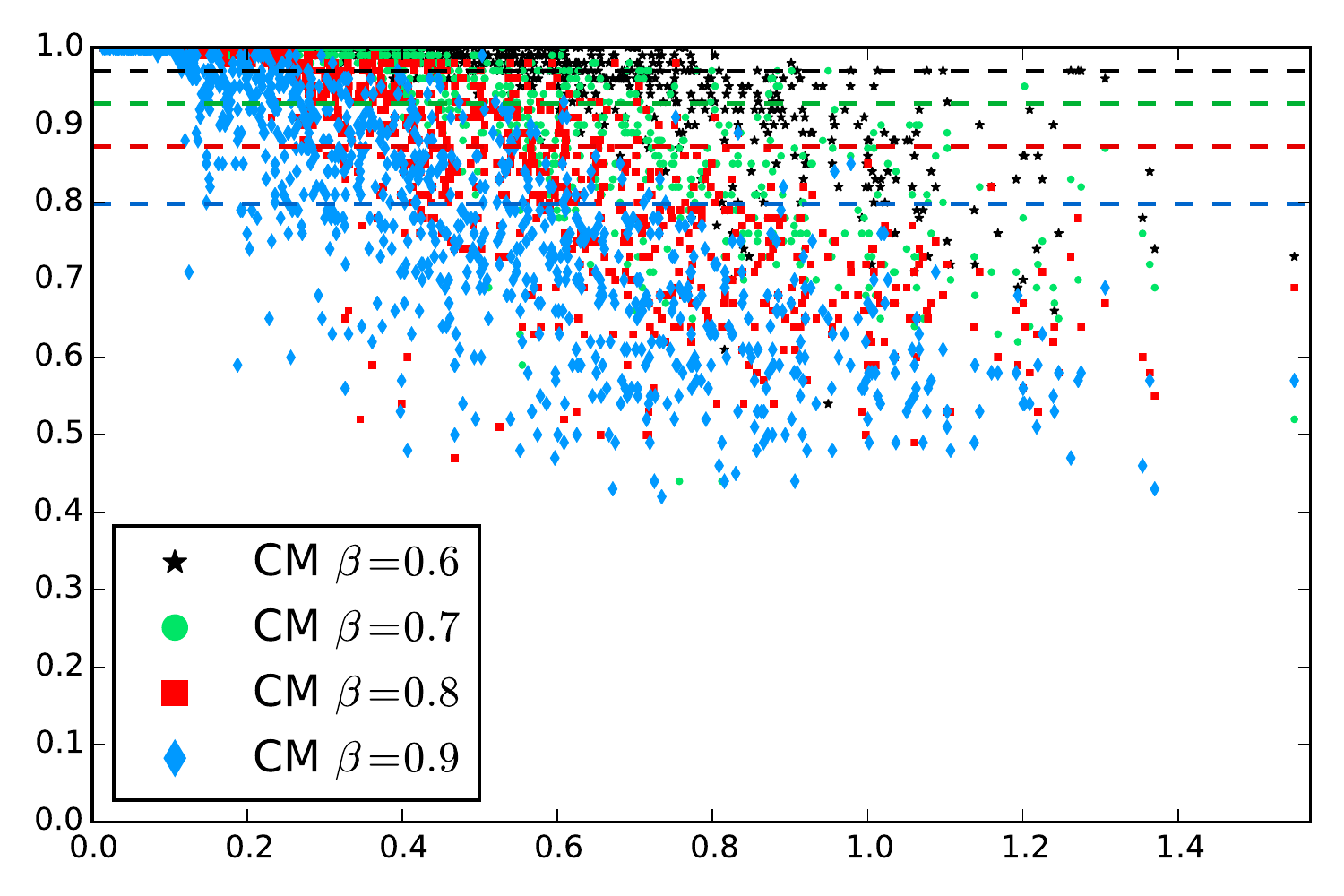}
	\caption{Best value of $\alpha$ with respect to the Friedrichs angle for 1000~pairs of random subspaces for AAMR (left) and CM (right). For each $\beta$, the average value of the best $\alpha$  is represented by a dashed line.}\label{fig:best_alpha}
\end{figure}

In Figure~\ref{fig:num_exp_alpha} we have plotted four prototypical examples of the number of iterations required by AAMR to find a solution for four different values of the Friedrichs angle. It can be clearly seen in this figure that the optimal value of $\alpha$ for the DRM in~\eqref{eq:DRM} is $0.5$, as it was expected (see~\cite[Remark~3.11(i)]{BCNPW15}). For this reason, we set the value of $\alpha$ for the DRM to $0.5$ in our subsequent experiments, and the value of $\alpha$ to $0.9$ for AAMR, based on the results shown in Figures~\ref{fig:best_alpha} and~\ref{fig:num_exp_alpha}. For CM, the best value of $\alpha$ seems to be more influenced by both the value of $\beta$ and the Friedrichs angle. Nonetheless, we set it to $\alpha=0.9$, as it appears to be a sensible choice as well.

\begin{figure}[ht!]\addtocounter{subfigure}{-1}
	\centering
	\subfigure{\includegraphics[width=0.89\linewidth]{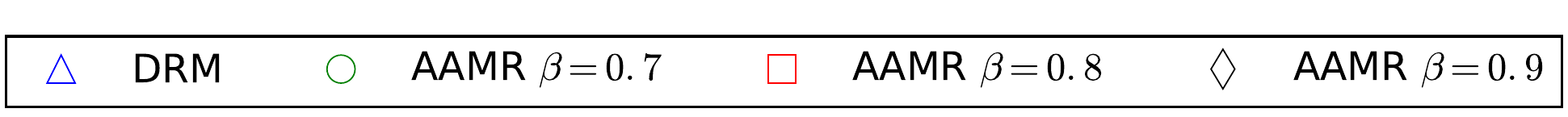}}
	\subfigure[Friedrichs angle: $0.041$ radians.]{\includegraphics[width=0.46\linewidth]{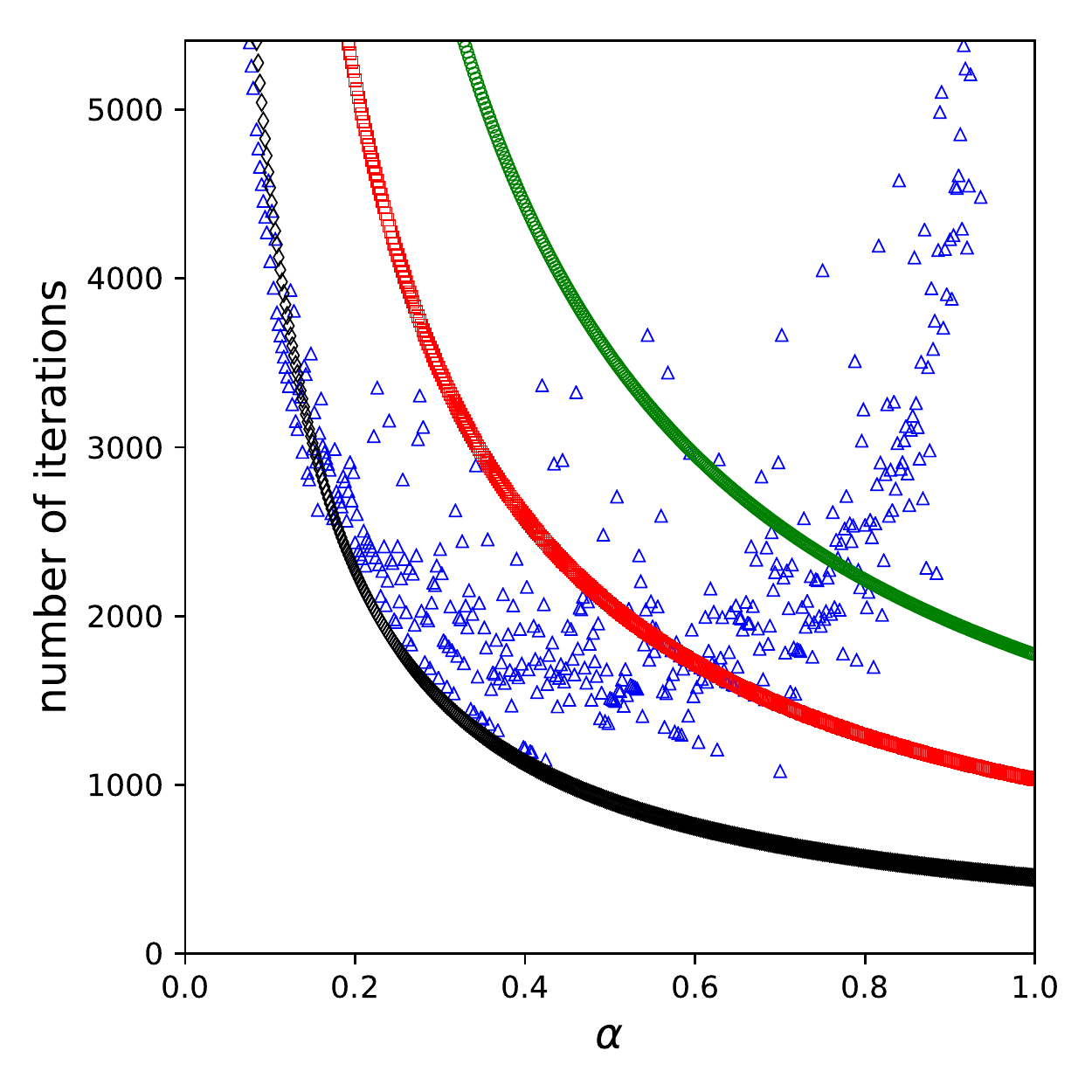}}
	\subfigure[Friedrichs angle: $0.283$ radians.]{\includegraphics[width=0.46\linewidth]{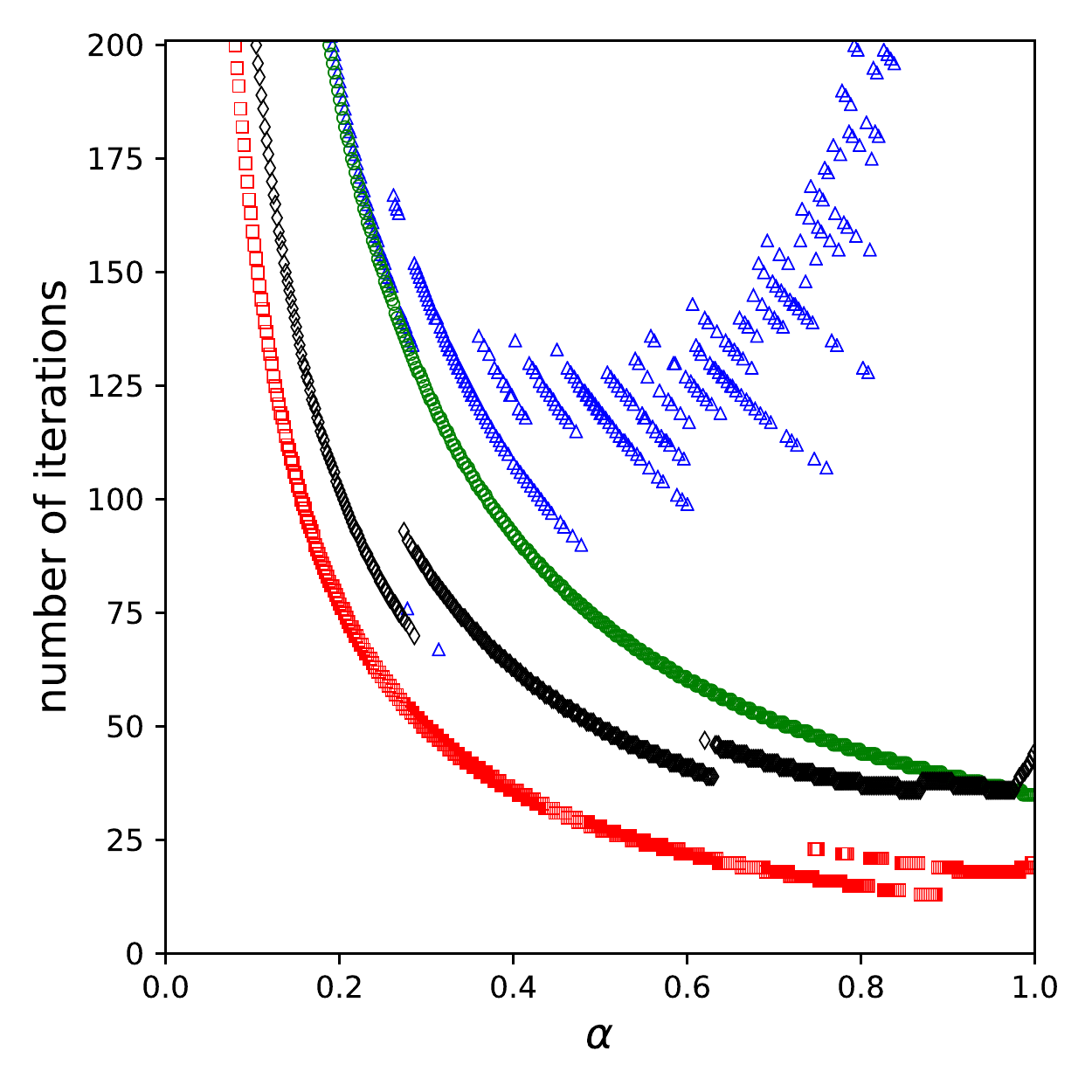}}
	\subfigure[Friedrichs angle: $0.61$ radians.]{\includegraphics[width=0.46\linewidth]{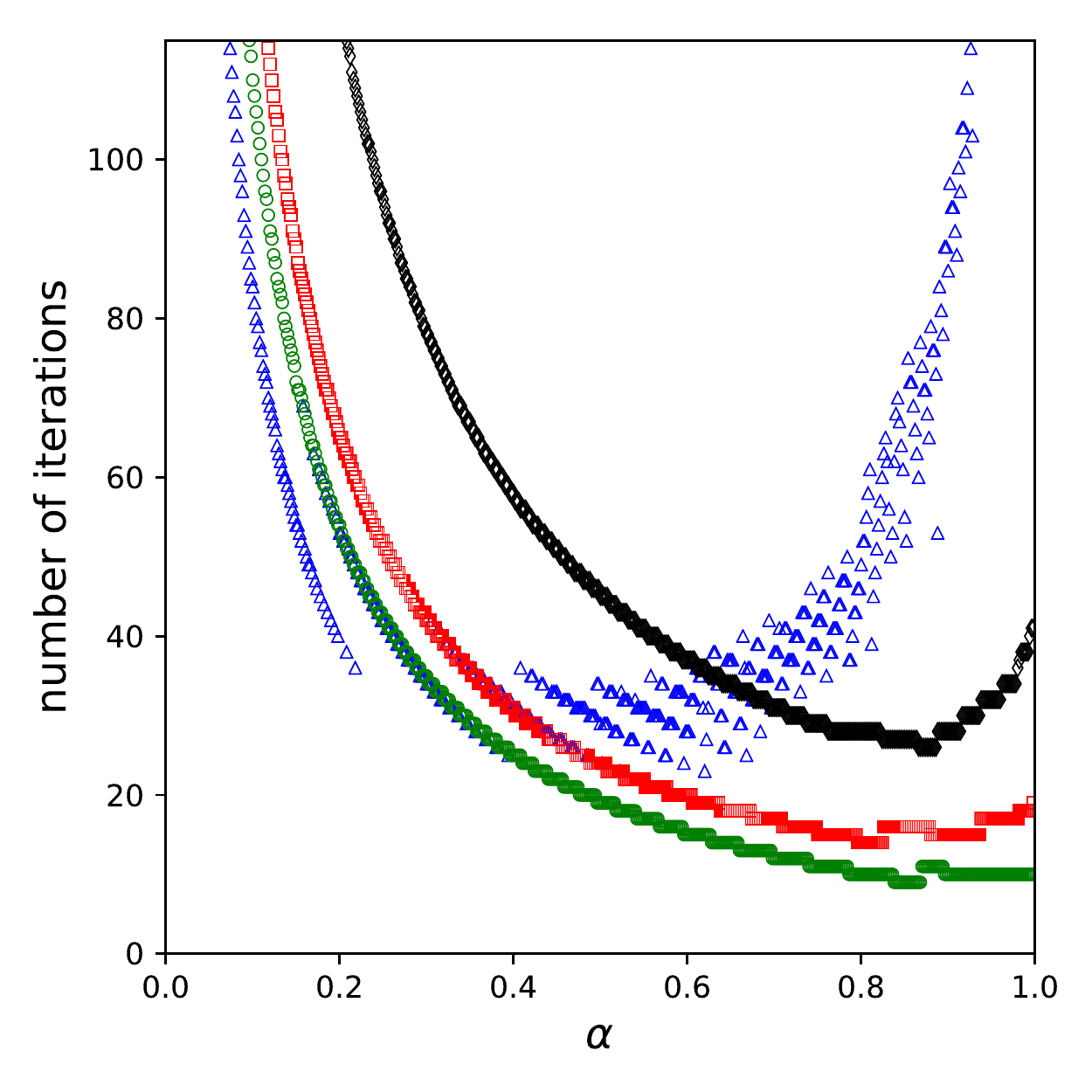}}
	\subfigure[Friedrichs angle: $1.001$ radians.]{\includegraphics[width=0.46\linewidth]{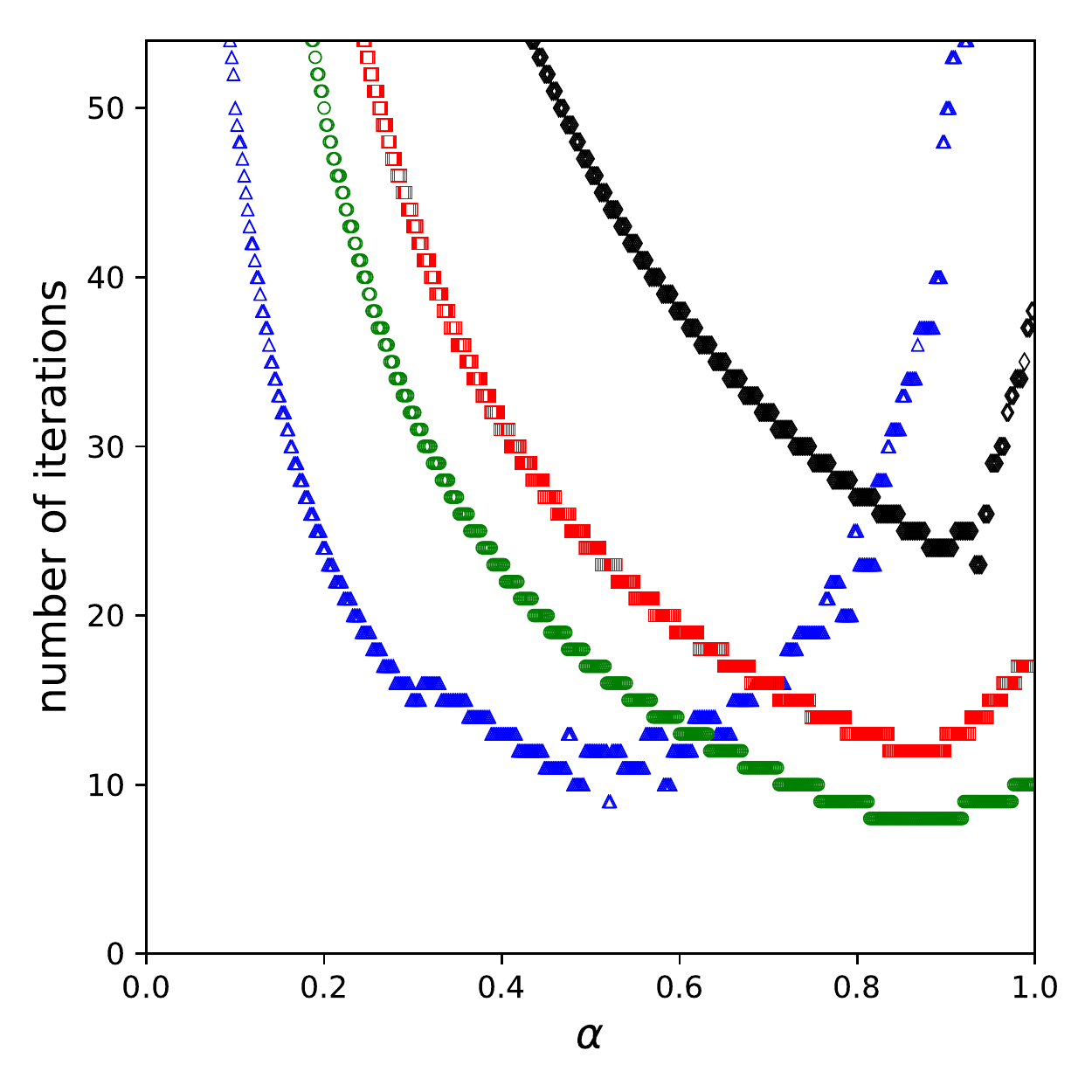}}
	\caption{Number of required iterations with respect to the value of $\alpha$ of DRM and AAMR for three different values of the parameter $\beta$.} \label{fig:num_exp_alpha}
\end{figure}

Our second experiment consisted in replicating some of the tests performed in~\cite{BCNPW14} to compare DRM and MAP, adding this time the results of the new AAMR method, CM and Haugazeau's method. We randomly generated $100$ pairs of subspaces $U$ and $V$ in $\R^{50}$. For each pair of subspaces, $10$~random starting points (with Euclidean norm 10) were chosen and each of the four methods were applied. As we realized that the parameter $\beta$ had a big influence in the behavior of the AAMR scheme, as can be observed in Figure~\ref{fig:num_exp_alpha}, we computed the sequences generated by the AAMR method for six different values of $\beta$ (these values were $0.5$, $0.6$, $0.7$, $0.8$, $0.9$ and $0.99$). We did the same with CM, setting also the value of $\alpha_n$ to $0.9$ in~\eqref{eq:combettes2}.
Although there is a freedom of choice for the initial point in the AAMR method and CM, we took it as the point to be projected, as this is the starting point that needs to be used by DRM, MAP and Haugazeau's method. The results are shown in Figures~\ref{fig:num_exp_median} and \ref{fig:num_exp_std}. For each pair of subspaces, the horizontal axis represents the Friedrichs angle, and the vertical axis represents the median (Fig.~\ref{fig:num_exp_median}) or the standard deviation (Fig.~\ref{fig:num_exp_std}) of the number of iterations required to converge for $10$~random initializations.

\begin{figure}
	\centering
	\subfigure{\includegraphics[width=0.75\linewidth]{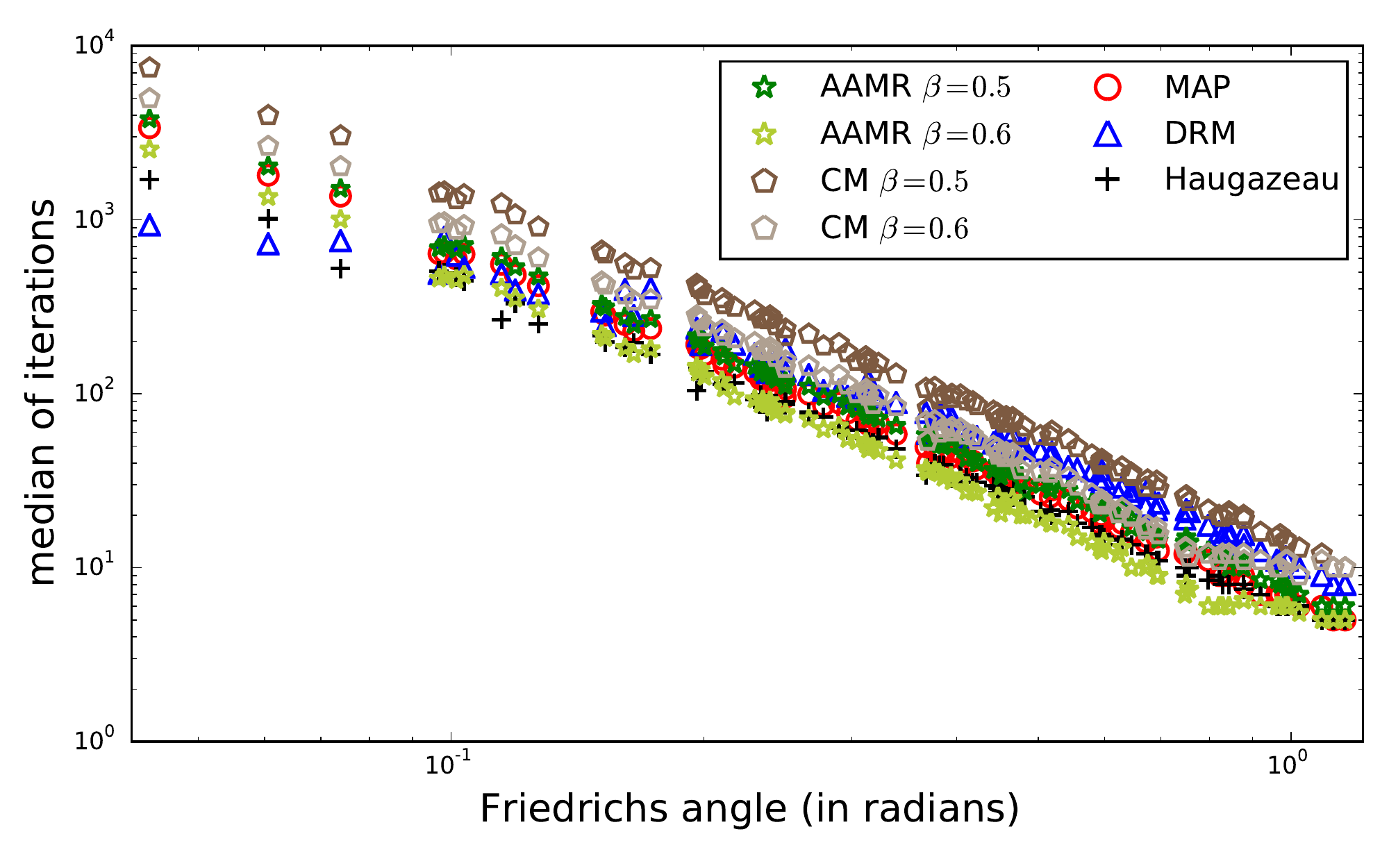}}
	\subfigure{\includegraphics[width=0.75\linewidth]{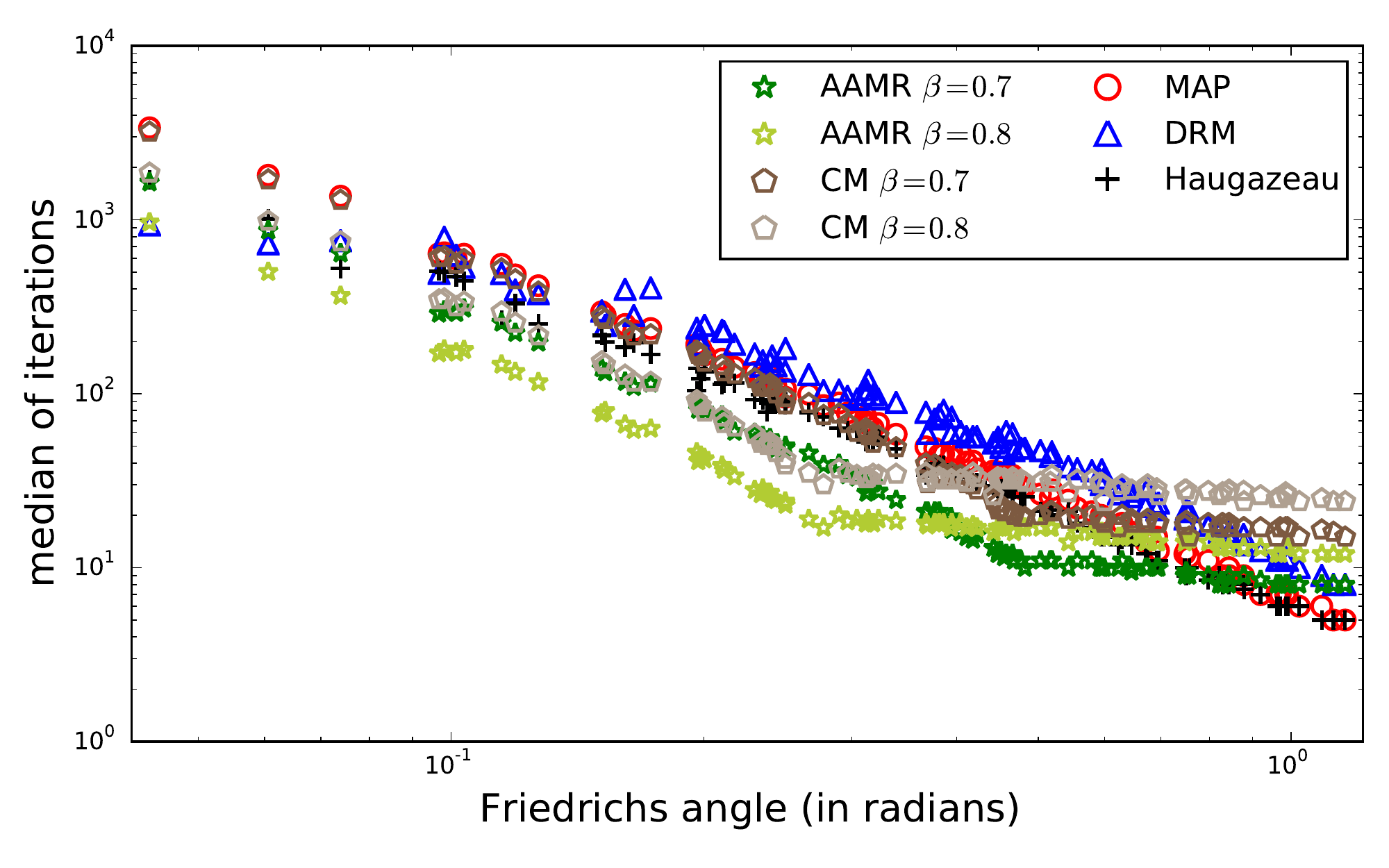}}
	\subfigure{\includegraphics[width=0.75\linewidth]{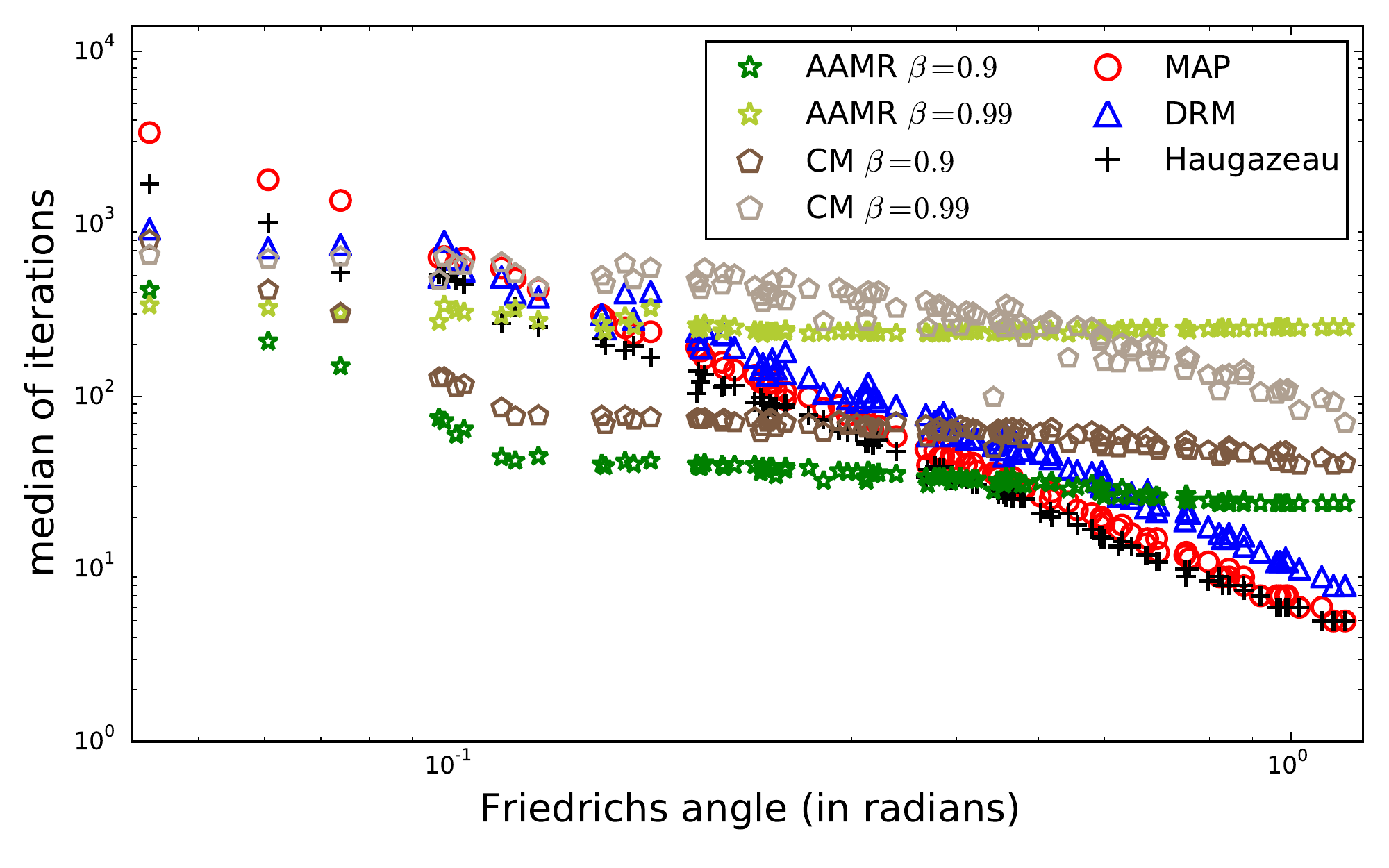}}
	\caption{Median of the required number of iterations with respect to the Friedrichs angle of MAP, DRM, Haugazeau's method, CM and AAMR for six different values of the parameter~$\beta$.} \label{fig:num_exp_median}
\end{figure}

\begin{figure}
	\centering
	\subfigure{\includegraphics[width=0.75\linewidth]{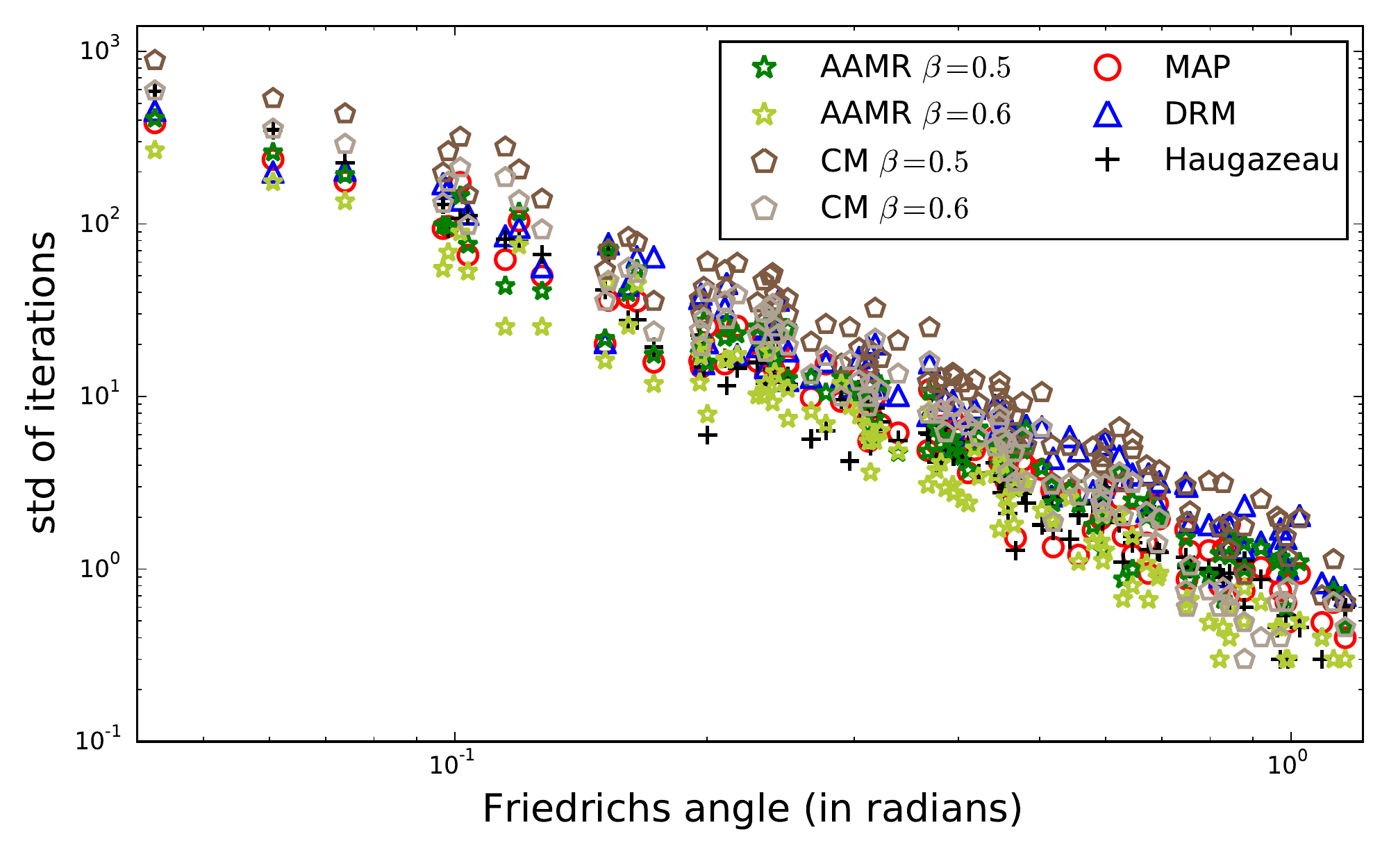}}
	\subfigure{\includegraphics[width=0.75\linewidth]{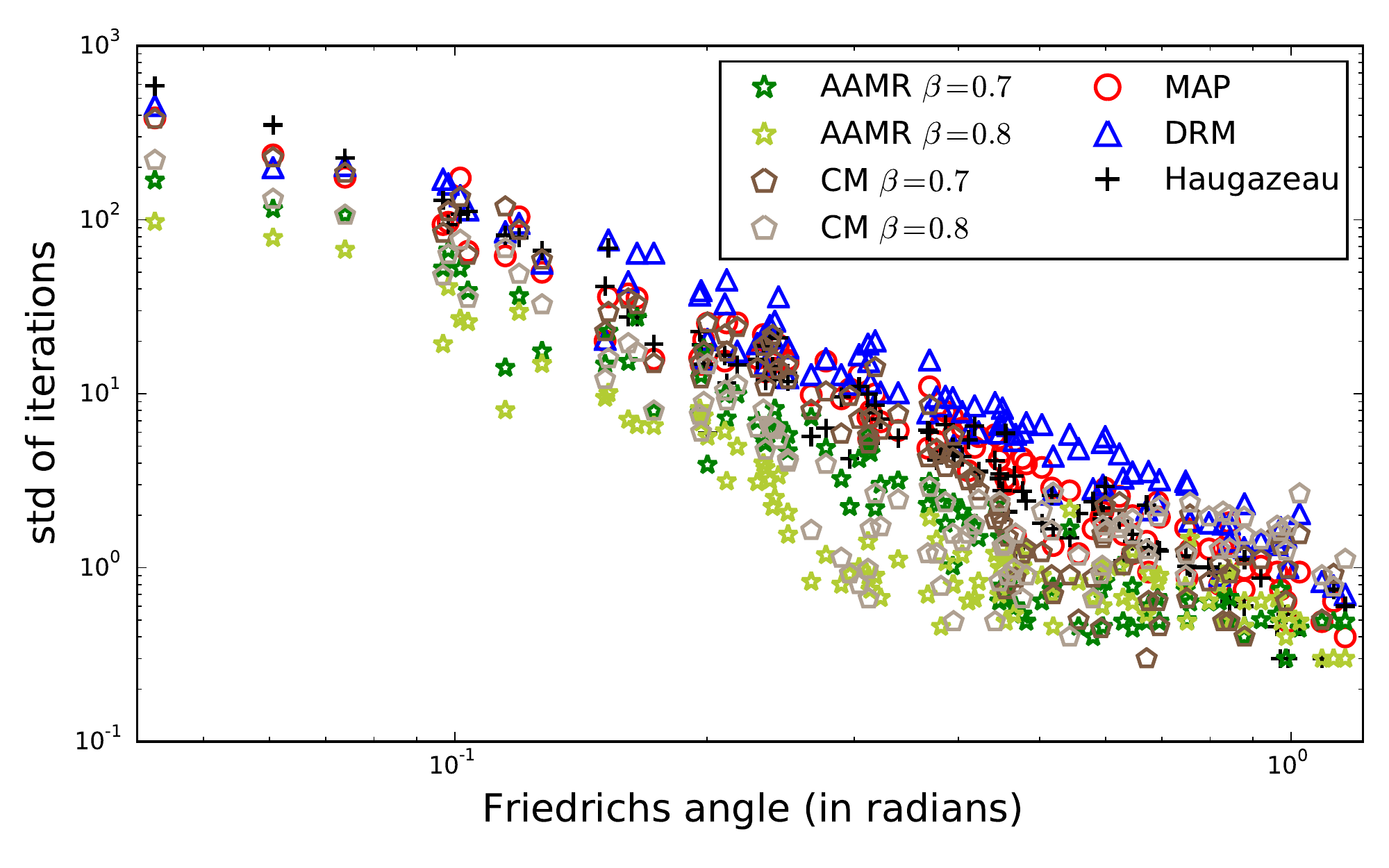}}
	\subfigure{\includegraphics[width=0.75\linewidth]{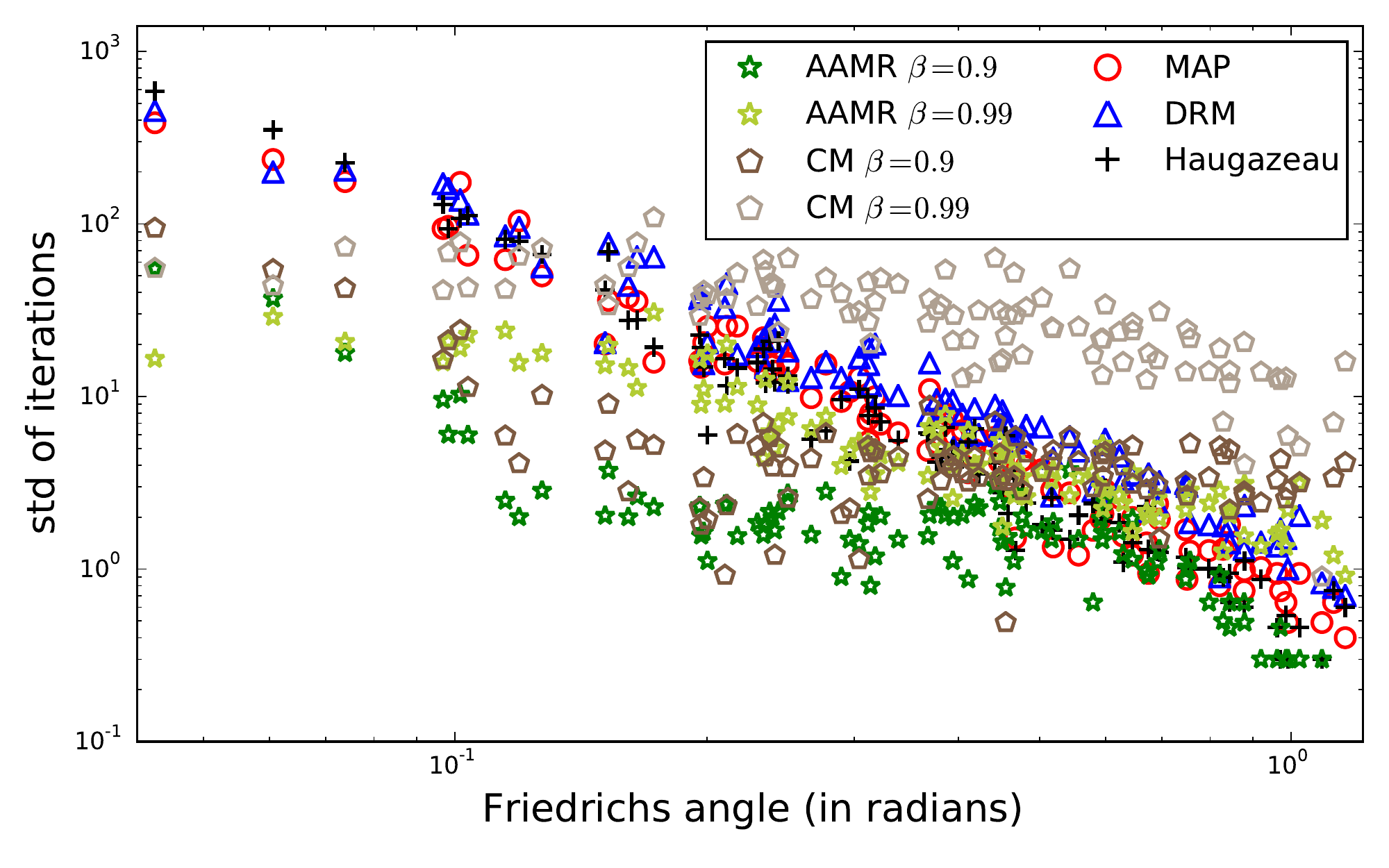}}
	\caption{Standard deviation of the required number of iterations with respect to the Friedrichs angle of MAP, DRM, Haugazeau's method, CM and AAMR for six different values of the parameter $\beta$.} \label{fig:num_exp_std}
\end{figure}

On one hand, we can deduce from Figure~\ref{fig:num_exp_median} that the rate of convergence of the AAMR method depends on both the angle and the parameter $\beta$. For values of $\beta$ above $0.7$, there exists an interval of small angles for which AAMR is the fastest method. For large angles, MAP and Haugazeau's method clearly outperforms DRM and AAMR. A simple example showing this behavior is depicted in Figure~\ref{fig:dykstra}.

We observe that MAP, DRM and Haugazeau's method satisfy a decrease in the number of iterations when the angle increases. Unfortunately, while the number of iterations in these three methods keep on decreasing for large values of the angle, the AAMR method and CM seem to have an asymptotic behavior around a horizontal line. That is, they need a minimum number of iterations to converge whatever the angle is (although this number is not very big). On the other hand, Figure~\ref{fig:num_exp_std} shows that the AAMR method is \emph{more robust} in terms of the standard deviation of the number of iterations. In fact, it seems that the larger the value of $\beta$ is, the more robust it becomes.

\begin{figure}
	\centering
	\subfigure[Small angle.]{\includegraphics[width=0.49\linewidth]{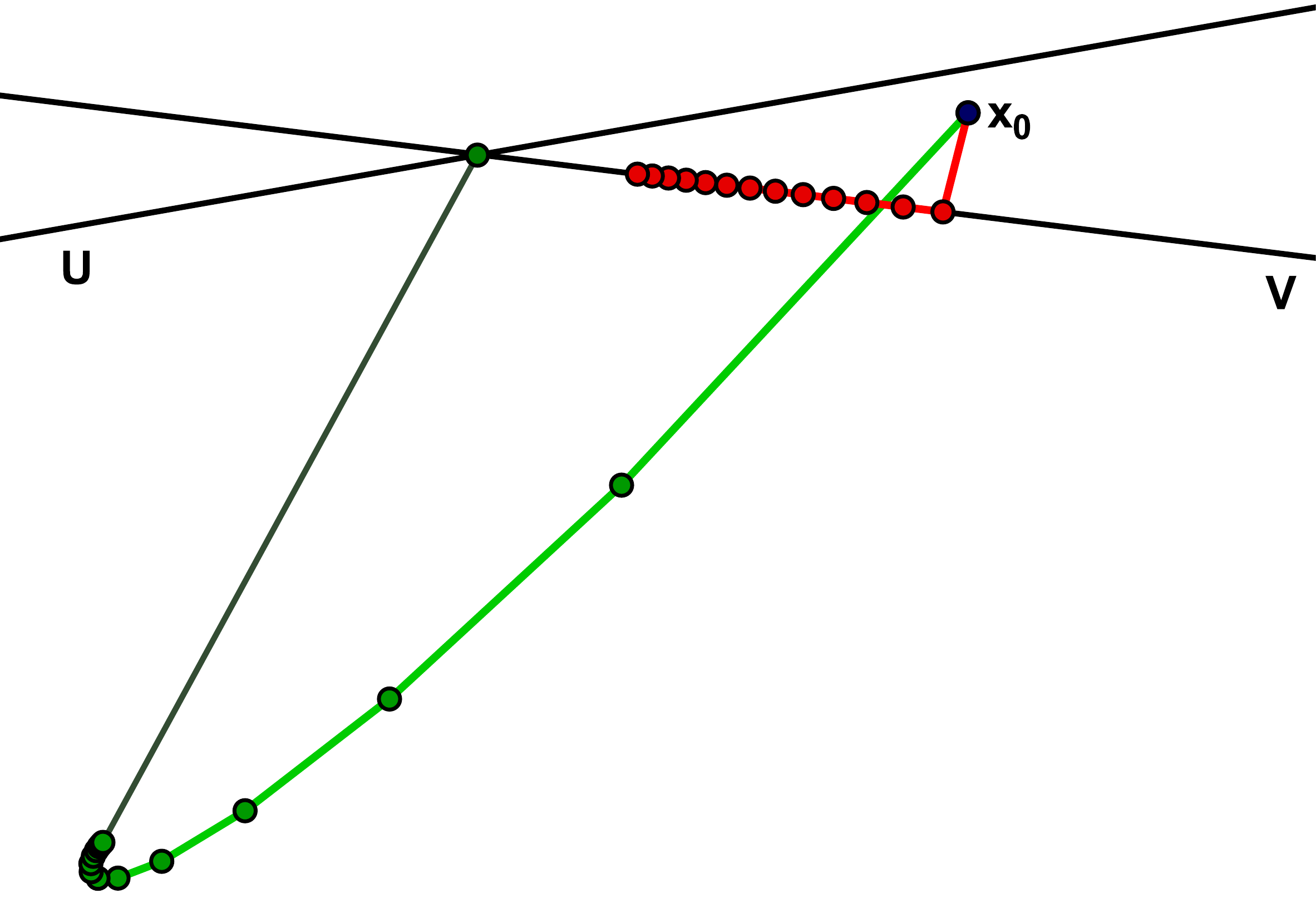}}
	\subfigure[Large angle.]{\includegraphics[width=0.49\linewidth]{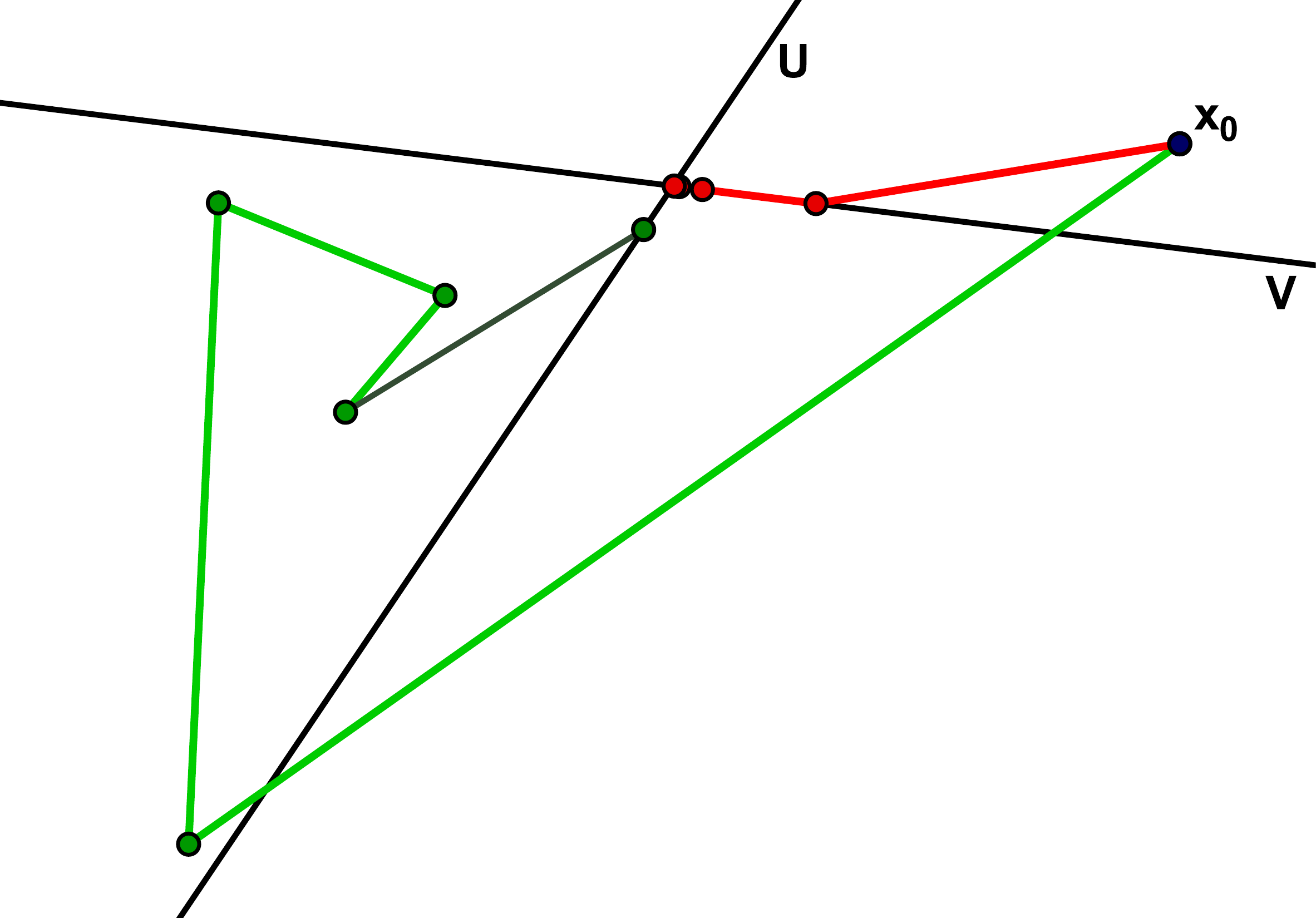}}
	\caption{Behavior of the AAMR (in green) and alternating projections (in red) algorithms when applied to two lines in $\R^2$ for two different Friedrichs angles. We see that AAMR outperforms MAP for small angles, while MAP is faster for large angles.} \label{fig:dykstra}
\end{figure}

With the purpose of additionally comparing the rate of convergence of the methods, we computed the distance of the first $100$ iterates of the sequences to be monitored to the real solution. We show the results of four instances with well-differentiated Friedrichs angles in Figure~\ref{fig:num_exp_dts}. To improve the clarity and comprehensibility of the plots, we have not included the results of CM as it was outperformed by AAMR. One might expect the AAMR method to inherit the ``rippling'' behavior of the DRM, specially when $\beta$ is large, which is when the definition of the iterations of both methods are more similar. This is not entirely truth: although the AAMR method indeed shows these ``waves'' in Figure~\ref{fig:num_exp_dts}, this behavior depends on both the Friedrichs angle and the parameter~$\beta$. Additionally, we see in this figure that the AAMR method with a large parameter $\beta$ clearly outperforms the other schemes when the Friedrichs angle is small. The larger the angle becomes, the better MAP and Haugazeau's method behave. As pointed out by Bauschke et al. in~\cite{BCNPW14}, it is expected that MAP performs better than DRM when the Friedrichs angle is large, as the rates of convergence of these two methods when applied to subspaces $U$ and $V$ are, respectively, $c_F(U,V)^2$ and $c_F(U,V)$. Finally, we clearly observe in Figure~\ref{fig:num_exp_dts} that HLWB is the slowest algorithm for solving this problem.

\begin{figure}[ht!]\addtocounter{subfigure}{-1}
	\centering
	\subfigure{\includegraphics[width=0.9\linewidth]{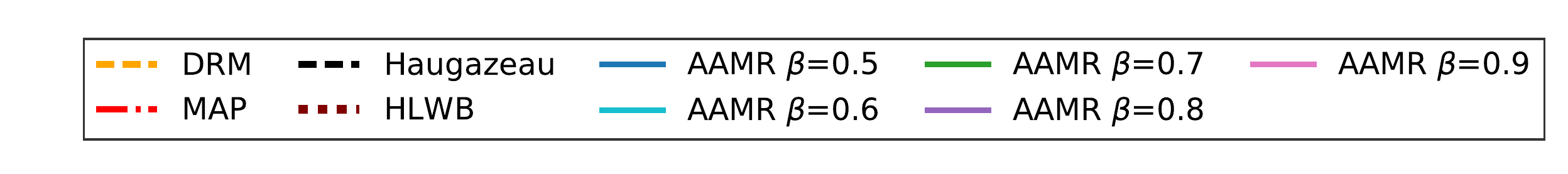}}
	\subfigure[Friedrichs angle: $0.074$ radians.]{\includegraphics[width=0.49\linewidth]{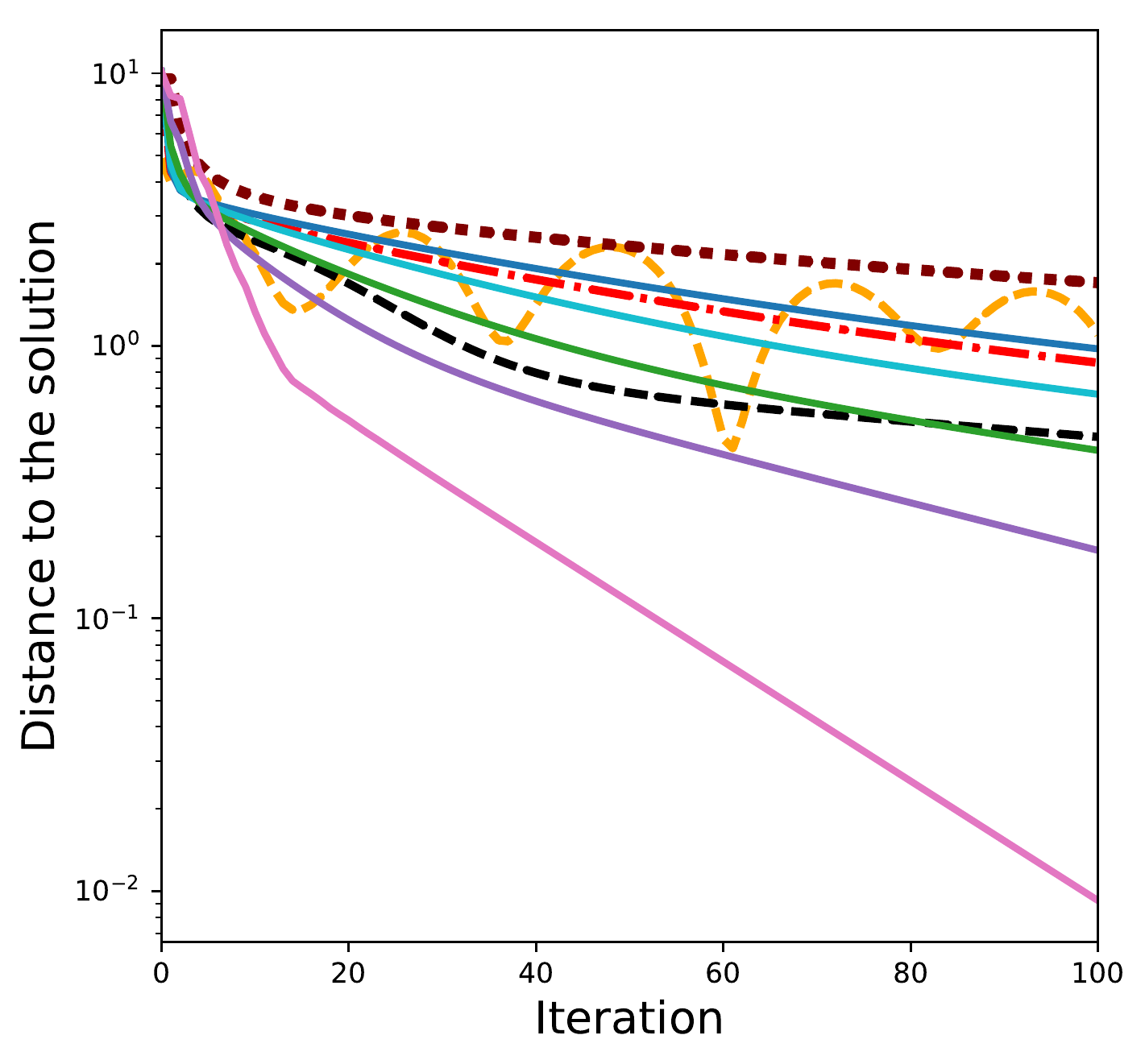}}
	\subfigure[Friedrichs angle: $0.314$ radians.]{\includegraphics[width=0.49\linewidth]{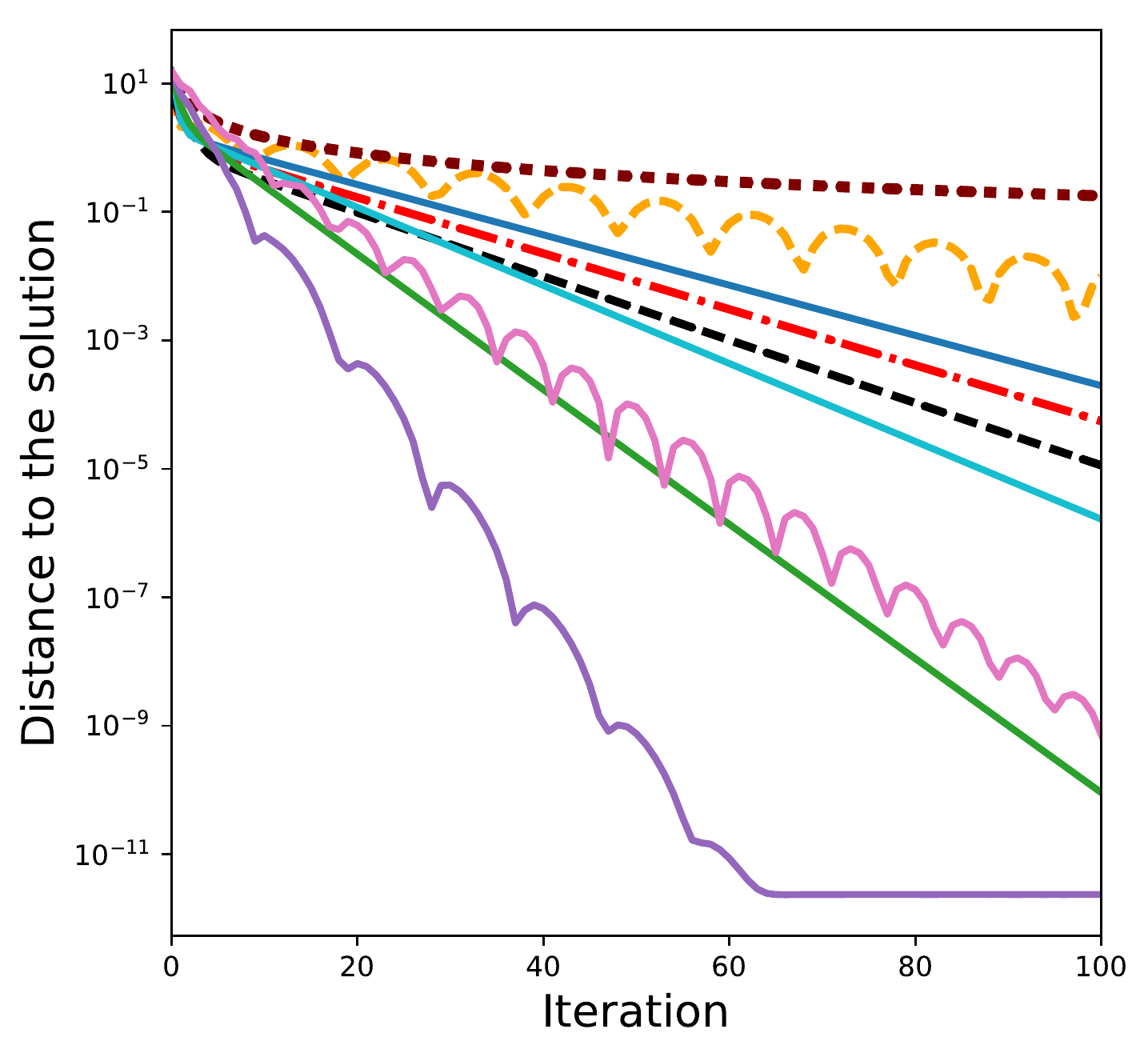}}
	\subfigure[Friedrichs angle: $0.751$ radians.]{\includegraphics[width=0.49\linewidth]{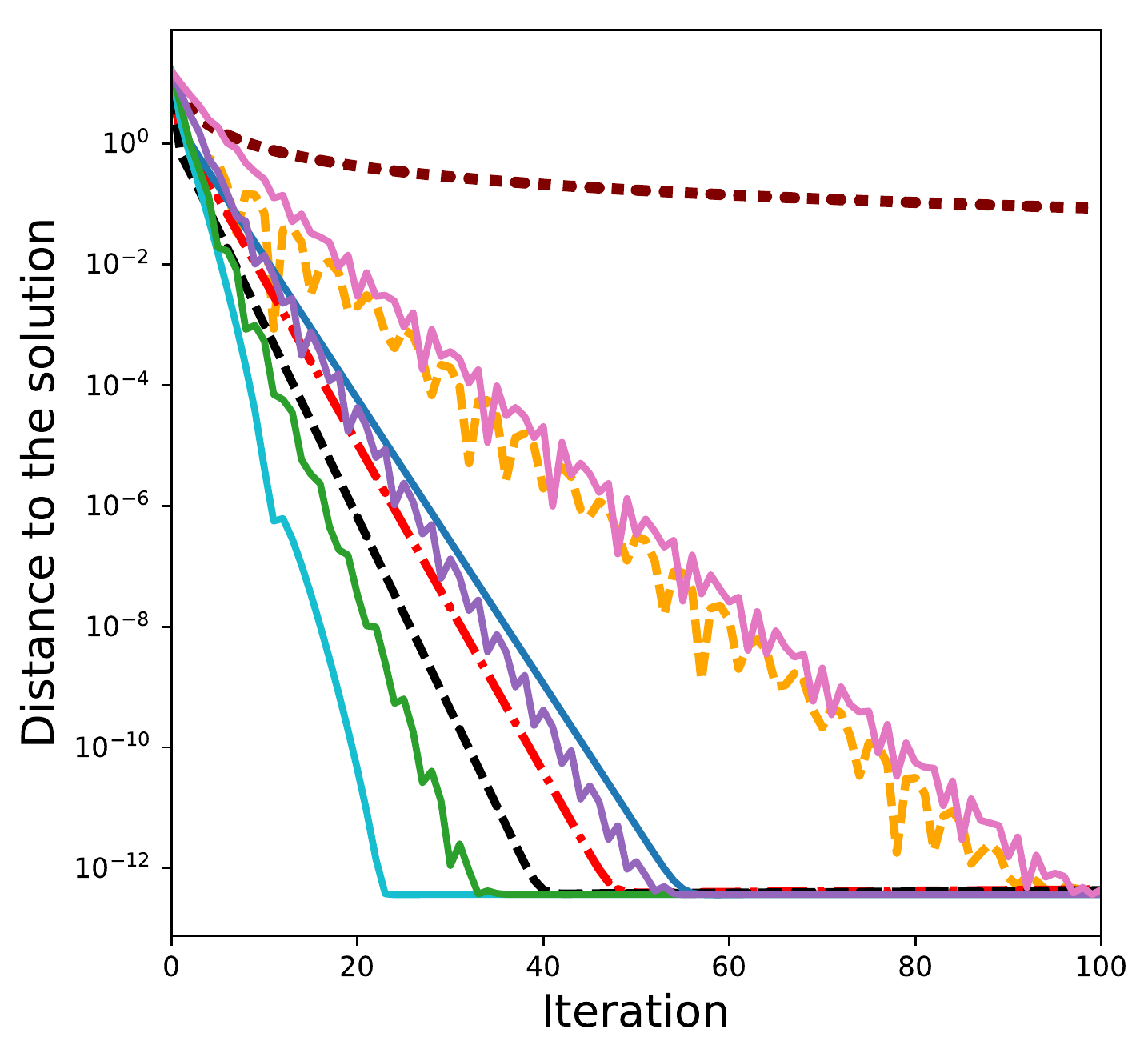}}
	\subfigure[Friedrichs angle: $1.367$ radians.]{\includegraphics[width=0.49\linewidth]{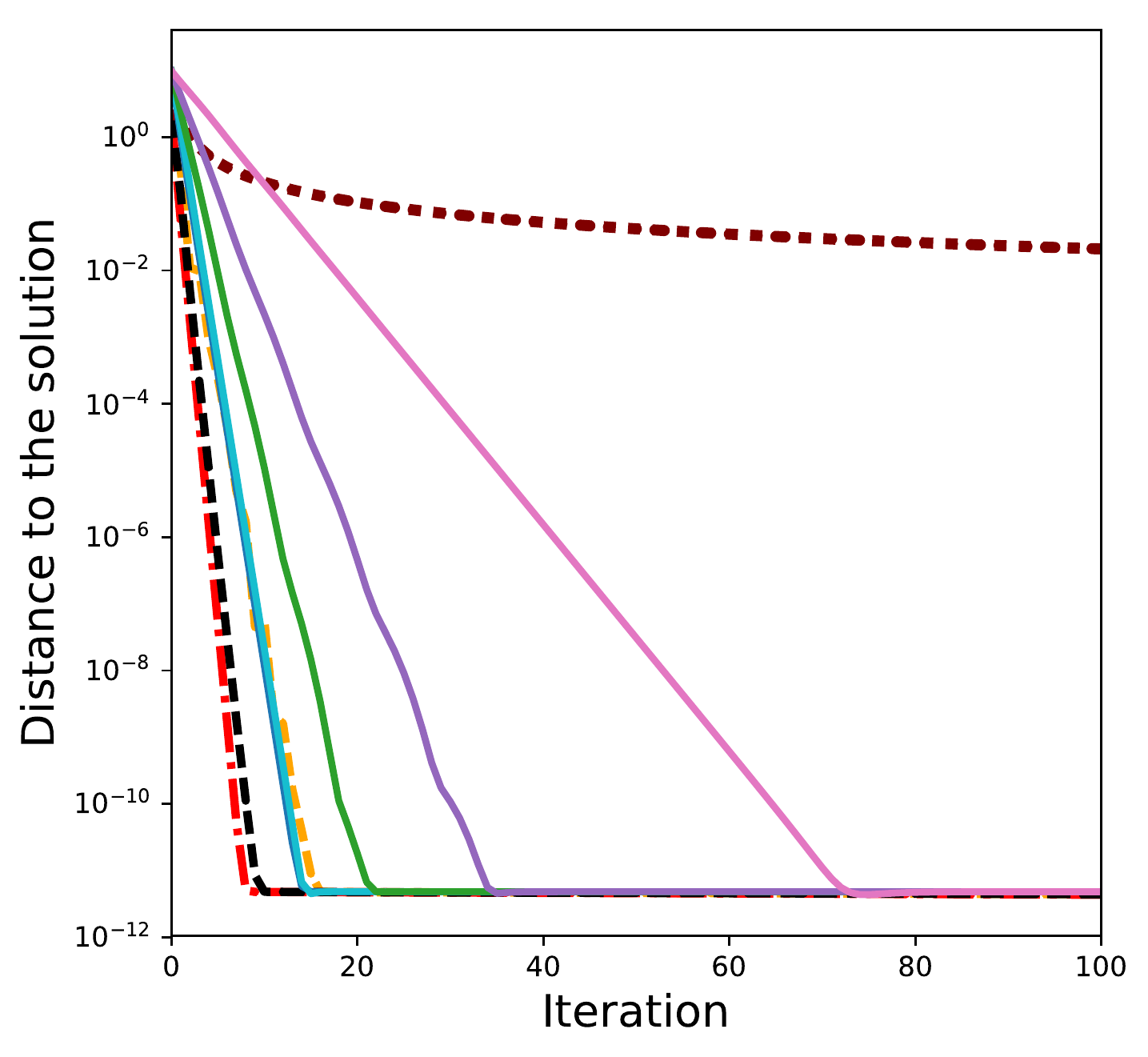}}
	\caption{Distance of the $100$ first iterations of the monitored sequences of MAP, DRM, HLWB, Haugazeau's method and AAMR for five different values of the parameter $\beta$ to the real solution.} \label{fig:num_exp_dts}
\end{figure}

In our third numerical experiment, we continued investigating how the parameter $\beta$ affects the number of iterations depending on the angle. In this experiment, $1000$ pairs of subspaces were generated. Then, for $10$ random starting points, we ran the AAMR method for every value of $\beta$ in $\{0.1, 0.2, 0.3, 0.4, 0.5, 0.6, 0.7, 0.8, 0.9, 0.99\}$. The results are shown in Figure~\ref{fig:num_exp_betas}. One can see that values of $\beta\leq0.4$ are an inefficient choice, since $\beta=0.5$ appears to dominate them for every angle. Larger values of $\beta$ work better for small angles, but the performance of the method  is worse for large angles for these large values of $\beta$.

\begin{figure}[ht!]
	\centering
	\includegraphics[width=1\linewidth]{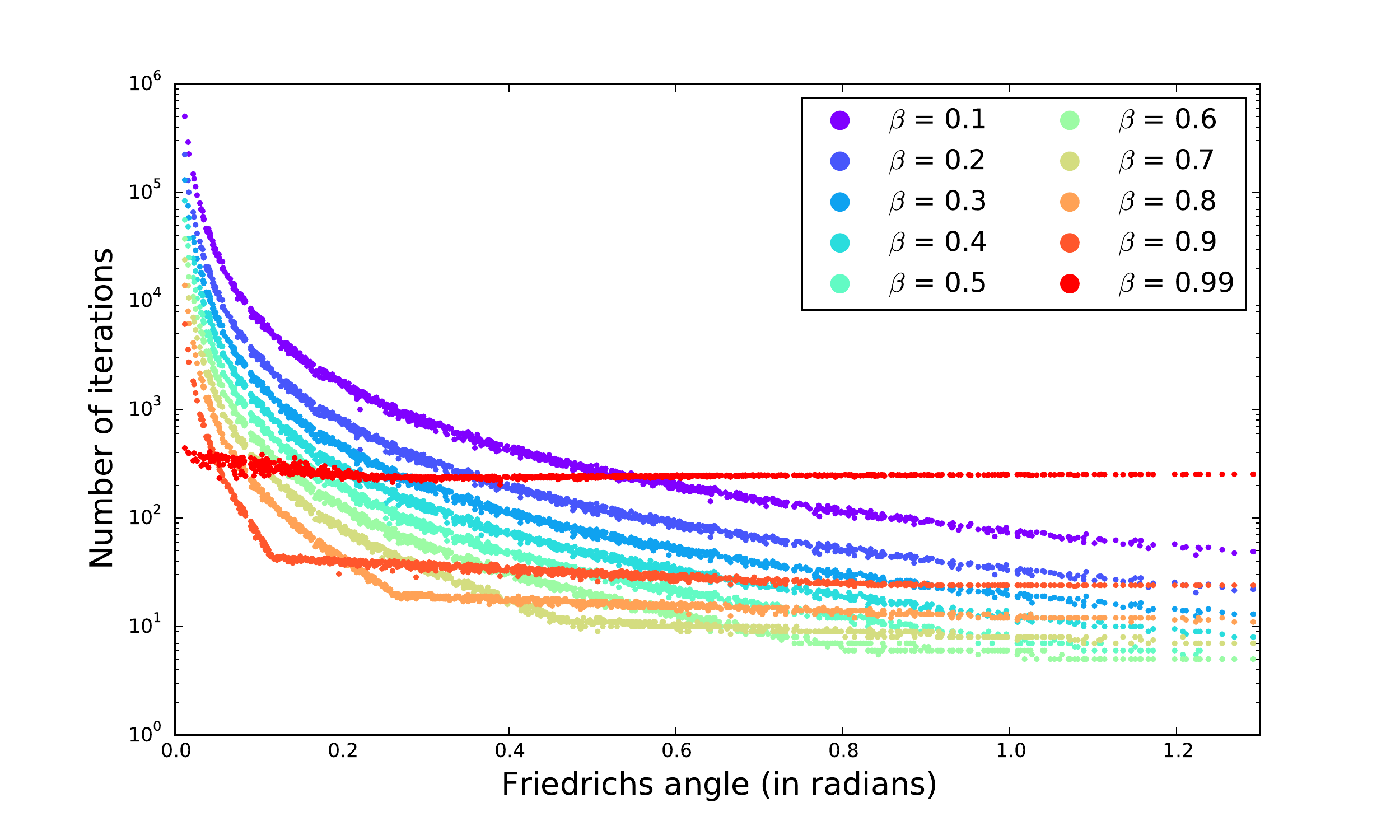}
	\caption{Median number of iterations for $10$ random starting points required by the AAMR method for different values of the parameter $\beta$ with respect to the Friedrichs angle.} \label{fig:num_exp_betas}
\end{figure}

In order to further analyze the influence of the parameter~$\beta$, we decided to test which~$\beta$ suits best each angle. In this fourth experiment we randomly generated $100$ pairs of subspaces in $\mathbb{R}^{50}$, choosing them so that their Friedrichs angles were \emph{approximately} equally distributed in $\left[0,\frac{\pi}{2}\right]$ (to this aim, we divided the interval $\left[0,\frac{\pi}{2}\right]$ into $100$ subintervals and then we randomly chose one pair of subspaces whose Friedrichs angle belongs to each subinterval).  To reduce the possible influence of any outlier, we randomly generated $100$ starting points (instead of 10) for each pair of subspaces and run the AAMR method for every $\beta$ in $\{0.4, 0.405, 0.41, \ldots, 0.985,0.99,0.995\}$. Then, for each pair of subspaces, we selected the $\beta$ in the latter set that minimizes the median number of iterations for the $100$ starting points. Observe that this makes a total of 1.2 million runs of the AAMR method. Figure~\ref{fig:num_exp_optimbeta} contains the results, where the optimal value of $\beta$ is represented in the vertical axis, while the Friedrichs angle of each pair of subspaces is represented in the horizontal axis. Additionally, we represented in the same figure the least squares quadratic and exponential fitting curves.

\begin{figure}[ht!]
	\centering
	\includegraphics[width=1\linewidth]{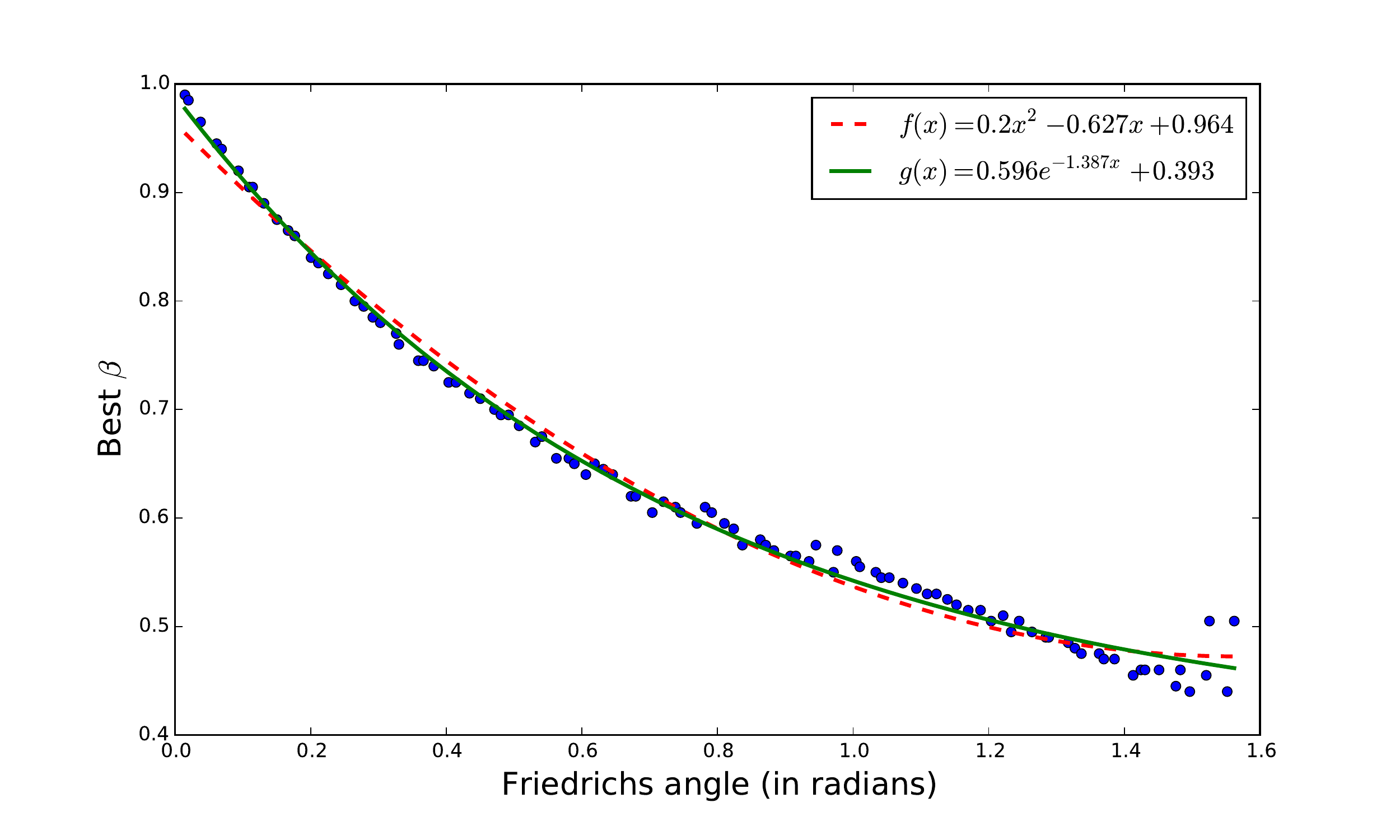}
	\caption{Optimal $\beta$ with respect to the median number of iterations of the AAMR method for $100$ subspaces with different Friedrichs angles.} \label{fig:num_exp_optimbeta}
\end{figure}

Finally, in Figure~\ref{fig:median_optimbeta}, we repeated the experiment shown in Figure~\ref{fig:num_exp_median}. This time we used the Friedrichs angle $\theta$ between each pair of subspaces to choose the value of the parameter $\beta$ in AAMR by using the exponential fitting curve obtained in Figure~\ref{fig:num_exp_optimbeta}. For a more fair comparison with the alternating projection method, we used the relaxed alternating projection method (RAP)
$$T_\mu:=(1-\mu) I+\mu P_V P_U,$$
with $\mu=\frac{2}{1+\sin^2\theta}$, which was shown in~\cite[Theorem~3.6]{BCNPW15} to be the parameter attaining the smallest convergence rate of the latter scheme. For Douglas--Rachford, the optimal parameter for every angle is always the classical one $\alpha=\frac{1}{2}$ (see~\cite[Remark~3.11]{BCNPW15}). Clearly, with the exception of some very large angles, AAMR outperforms the other methods.

\begin{figure}[ht!]
	\centering
	\includegraphics[width=0.95\linewidth]{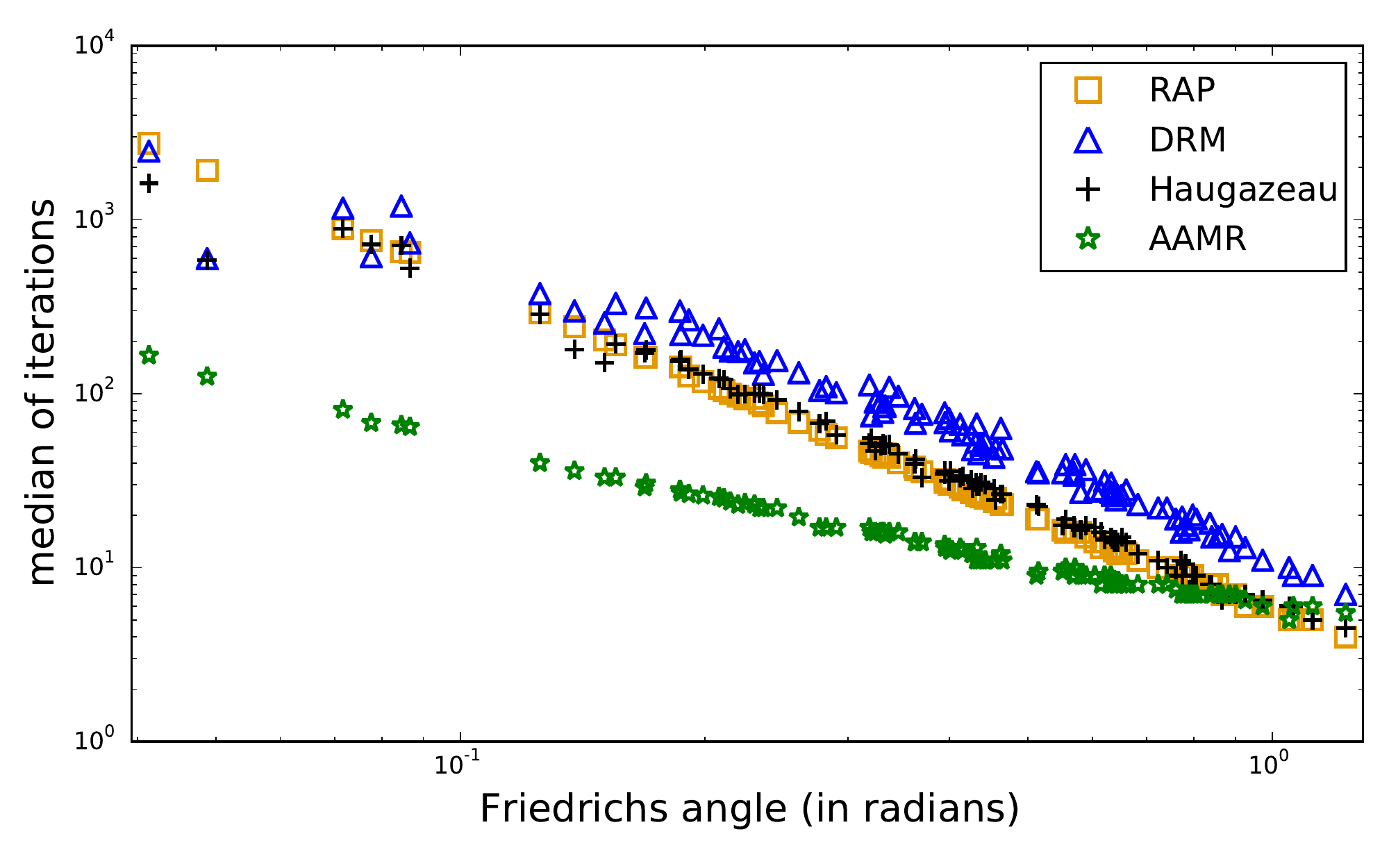}
	\caption{Median of the required number of iterations with respect to the Friedrichs angle of the relaxed alternating projection method with optimal parameter $\mu=\frac{2}{1+\sin^2\theta}$, Douglas--Rachford, Haugazeau's method and AAMR for $\alpha=0.9$ and $\beta=g(\theta)=0.596 e^{-1.387\theta}+0.393$, where $\theta$ is the Friedrichs angle between each pair of subspaces.} \label{fig:median_optimbeta}
\end{figure}

All these experiments led us to recommend a choice of $\alpha=0.9$ and $\beta\in[0.7,0.8]$ for general problems, as it seems to give good convergence results for both small and large angles. Probably, a scheme adapting the value of the parameter $\beta$ would be the best option.

\section{Conclusions and future work}\label{sec:conclusion}

A new projection scheme for solving the best approximation problem, the averaged alternating modified reflections (AAMR) method, was introduced and studied. Even though each iteration of a AAMR method is very similar to the classical Douglas--Rachford method (DRM), the AAMR scheme yields a solution to the best approximation problem, unlike the DRM, which only gives a point in the intersection of the sets. Under a constraint qualification, the method was proved to be strongly convergent to the solution to the best approximation problem. The numerical experiments performed to find the closest point in the intersection of two subspaces show that the new AAMR method outperforms the classical method of alternating projections, the Douglas--Rachford method and Haugazeau's method, when the Friedrichs angle between the subspaces is small. These experiments also show that a choice of the parameter $\alpha=0.9$ and $\beta$ between $0.7$ and $0.8$ might be adequate for general purposes. Although the numerical tests we performed are promising, they are far from a complete computational study. This motivate us to further analyze the rate of convergence of the AAMR method in a future work, both numerically and analytically.

All the results in this work were done for closed and convex sets. Over the past decade, the Douglas--Rachford method has proven to be very effective in some highly non-convex settings~\cite{ABglobal,ABTmatrix,ABTcomb,ABT16,BKroad,BNlocal,benoist,BS11,HLnonconvex,Plinear}.
Because of the similarity of the AAMR scheme and the Douglas--Rachford method, it would be interesting to explore whether it would be possible to use the AAMR method as heuristic on non-convex feasibility problems, either alone, or combined with the DRM to avoid possible cycles. We believe that the best choice would be to use a scheme where the parameter $\beta\in\,]0,1]$ is changed when the method does not give an \emph{adequate} progress.

\paragraph{Acknowledgements} The authors thank Heinz Bauschke for his careful reading of a previous version of this paper, and for making various perceptive comments and suggestions. We also thank D. Russell Luke for his insightful comments. We are indebted to one of the referees for a number of constructive suggestions and for pointing us to reference~\cite{C09}, which led us to prove strong convergence of the shadow sequence in Theorem~\ref{theorem:NPM_convergence}. This work was partially supported by MINECO of Spain  and  ERDF of EU, grant MTM2014-59179-C2-1-P. F.J. Arag\'on Artacho was supported by the Ram\'on y Cajal program by MINECO of Spain  and  ERDF of EU (RYC-2013-13327) and R.~Campoy was supported by MINECO of Spain and ESF of EU (BES-2015-073360) under the program ``Ayudas para contratos predoctorales para la formaci\'on de doctores 2015''.

\end{document}